\documentclass[11pt,a4paper,twoside]{amsart}
\usepackage{amsmath,amsfonts,amssymb,amsthm}
\usepackage{algorithm,algpseudocode,color,multirow,slashbox,textcomp,times} 
\usepackage{arydshln,empheq,enumitem,fancybox,footnote,framed,graphics,lscape,psfrag,subfigure}
\usepackage[colorlinks=true]{hyperref}
\usepackage[a4paper]{geometry}
\usepackage{tikz}
\usetikzlibrary{calc}

\newtheorem{theorem}{Theorem}[section]
\newtheorem{lemma}[theorem]{Lemma}
\newtheorem{remark}[theorem]{Remark}
\newtheorem{definition}[theorem]{Definition}
\newtheorem{proposition}[theorem]{Proposition}
\newtheorem{example}{Example}
\newtheorem{case}{Case}

\newlength{\picw}

\newcommand{\dsR}{\mathbb{R}}
\newcommand{\dx}{\;dx}
\DeclareMathOperator{\dvg}{div}
\DeclareMathOperator{\spn}{span}
\newcommand{\caA}{\mathcal{A}}
\newcommand{\caI}{\mathcal{I}}
\newcommand{\caO}{\mathcal{O}}
\newcommand{\caS}{\mathcal{S}}
\newcommand{\cd}{\;}

\title{Guaranteed and Sharp a Posteriori Error Estimates \\
in Isogeometric Analysis
}

\author{Stefan~K.~Kleiss}
\address[Stefan~K.~Kleiss]{Johannes Kepler University, Altenberger Strasse 69, A-4040 Linz, Austria}
\email{stefan.kleiss@jku.at}
\author{Satyendra~K.~Tomar}
\address[Satyendra~K.~Tomar]{Dornacher Strasse 6/21, A-4040 Linz, Austria}
\email[Corresponding author]{tomar.sk.prof@gmail.com}

\thanks{This work was started when both the authors were employed in RICAM, Altenberger Strasse 69, A-4040 Linz, Austria.}

\date{September 8, 2014, ~(First version April 21, 2013)}

\keywords{Isogeometric analysis; B-splines and NURBS; A posteriori error estimates}
\subjclass{65N15, 65N30}

\begin{document}

\begin{abstract}
We present functional-type a posteriori error estimates in isogeometric analysis. These estimates, derived on functional grounds, provide guaranteed and sharp upper bounds of the exact error in the energy norm. {Moreover, since these estimates do not contain any unknown/generic constants, they are fully computable, and thus provide quantitative information on the error.} By exploiting the properties of non-uniform rational B-splines, we present efficient computation of these error estimates. The numerical realization and the quality of the computed error distribution are addressed. The potential and the limitations of the proposed approach are illustrated using several computational examples.
\end{abstract}

\maketitle

\section{Introduction}\label{sec_intro}

The geometry representations in finite element methods (FEM) and computer aided design (CAD) have been developed independent of each other, and are optimized for the purposes within their respective fields. As a consequence, the representations are different from each other, and a transfer of geometry information from CAD to FEM programmes (and vice versa) requires a transformation of geometry data. These transformations are, in general, not only costly, but also prone to approximation errors, and may require manual input.

\emph{Isogeometric analysis} (IGA), introduced by Hughes et al.\ \cite{Hughes05_4135}, see also \cite{Coott09_IGA}, aims at closing this gap between FEM and CAD. The key observation is that it is a widespread standard in CAD to use geometry representations based on non-uniform rational B-splines (NURBS), and that these NURBS basis functions have properties which make them suitable as basis functions for FEM. Instead of transforming the geometry data to a conventional FEM representation, the original geometry description is used directly, and the underlying NURBS functions are used as basis for the discrete solution. This way, the geometry is represented \emph{exactly} in the sense that the geometry obtained from CAD is not changed. Thus, the need for data transformation is eliminated, and furthermore, the exact representation from the coarsest mesh is preserved throughout the refinement process.
IGA has been thoroughly studied and analyzed (see, e.g.,  \cite{Bazi06_1031,Beirao05_271,Cott07_4160,Hughes10_301,Takacs11_3568}), and its potential has been shown by successful applications to a wide range of problems (see, e.g., \cite{Bazi07_173,Bazi08_3,Buffa10_1143,Elgu08_33,Niel2011_3242}).

As mentioned above, the most widely used {spline} representations in CAD are based on NURBS. The straightforward definition of NURBS basis functions leads to a tensor-product structure of the basis functions, and thus of the discretization. Since naive mesh refinement in a tensor-product setting has global effects, the development of local refinement strategies for isogeometric analysis is a subject of current active research. Such local refinement techniques include, for example, T-splines \cite{Bazi10_229, LiEtAl12_63, Scott11_126, Scott12_206, Sede04_276}, truncated hierarchical B-splines (THB-splines) \cite{Giannelli2012_485, GiannelliJS-14}, polynomial splines over hierarchical T-meshes (PHT-splines) \cite{Deng2008_76, Wang11_1438}, and locally-refineable splines (LR-splines) \cite{Dokken2013_331, JohannessenKD-14}.
The issue of adaptive, local refinement is closely linked to the question of efficient a posteriori error estimation (see, e.g., \cite{Ains00_book, Repin08_book} for a general overview on error estimators). In the light of adaptive refinement, an error estimator has to identify the areas where further refinement is needed due to the local error being significantly larger than in the rest of the domain. Hence, an accurate indication of the error distribution is essential.
{Another important objective in computing a posteriori error estimates is to address the \emph{quality assurance}, i.e., to quantify the error in the computed solution with certain degree of \emph{guarantee}.}
However, a posteriori error estimation in isogeometric analysis is still in an infancy stage. To the best of the authors' knowledge, the only published results are \cite{DedeS-12, Doerfel10_264, Johannessen11, KuruVZB-14, Vuong11_3554, Wang11_1438, Xu11_2021, XuMDG-13, ZeeV-11}.
A posteriori error estimates based on hierarchical bases, proposed by Bank and Smith \cite{BankSmith_93}, have been used in \cite{Doerfel10_264, Vuong11_3554}. The reliability and efficiency of this approach is subjected to the saturation assumption on the (enlarged) underlying space and the constants in the strengthened Cauchy inequality. As the authors remarked, the first assumption is critical and its validity depends on the considered example. Moreover, an accurate estimation of constants in the strengthened Cauchy inequality requires the solution of generalized minimum eigenvalue problem.
{As noted in \cite[Page 41]{Johannessen11}, this approach delivers \emph{less than satisfactory} results.}
Residual-based a posteriori error estimates have been used in \cite{Johannessen11, Wang11_1438, Xu11_2021, XuMDG-13}. This approach requires the computation of constants in Clement-type interpolation operators. Such constants are mesh (element) dependent, often generic/unknown or incomputable for general element shape; and the global constant often over-estimates the local constants, and thus the exact error.
{This fact has been explicitly stated by the authors in \cite[Pages 42-43]{Johannessen11} and in \cite[Remark 1]{Wang11_1438}}.
Goal-oriented error estimation approach has been studied in \cite{DedeS-12, KuruVZB-14, ZeeV-11}. The results presented in these studies show that neither the estimates of this approach are \emph{guaranteed} to be an upper bound, nor the efficiency indices of the estimates are sharp. Moreover, this approach also requires the solution of an adjoint problem, the cost of which can not be entirely neglected.
The approach of Zienkiewicz-Zhu type a posteriori error estimates is based on post-processing of approximate solutions, and depend on the superconvergence properties of the underlying basis. To the best of authors' knowledge, superconvergence properties for B-splines (NURBS) functions are not yet known.
Summarily, in general situations, the reliability and efficiency of these methods often depend on undetermined constants, which is not suitable for quality assurance purposes.
In this paper, we present \emph{functional-type a posteriori error estimates} for isogeometric discretizations. These error estimates, which were introduced in \cite{Repin97_201, Repin99_4311, Repin00_481} and have been studied for various fields (see \cite{Repin08_book} and the references therein), provide guaranteed, sharp and fully computable bounds (without any generic undetermined constants).
These estimates are derived on purely functional grounds (based on integral identities or functional analysis) and are thus applicable to any conforming approximation in the respective space. 
For elliptic problems with the weak solution $u\in H_0^1(\Omega)$, these error bounds involve computing an auxiliary function $y \in H(\Omega,\dvg)$.
{In order to get a sharp estimate, this function $y$ is computed by solving a \emph{global} problem. This could be perceived as a drawback when compared to error estimation techniques which rely on local computations and are thus apparently cheaper.
However, as briefly explained above, our emphasis is not only on adaptivity, but also on \emph{quantifying the error in the computed solution} (and thus guaranteeing the quality of the computed solution).
Therefore, the associated cost should be weighed against the stated objectives. To the best of authors' knowledge, there is no other, particularly cheaper, method available which can fulfill these objectives in general situations.
In this paper, we will elaborate how such estimates can be computed efficiently by a proper set-up of the global problem.}
Two aspects motivate the application of functional-type error estimates in IGA.
Firstly, unlike the standard Lagrange basis functions, NURBS basis functions of degree $p$ are, in general, globally $C^{p-1}$-continuous. Hence, NURBS basis functions of degree $p\geq 2$ are, in general, at least $C^1$-continuous, and therefore, {their gradients are} automatically in $H(\Omega,\dvg)$. Thereby, we avoid constructing complicated functions in $H(\Omega,\dvg)$, in particular for higher degrees (see, e.g., \cite{Buffa11_1407,BuffaEtAl11_818,EvaEtAl13_671}).
Secondly, since the considered problem is solved in an isogeometric setting, an efficient implementation of NURBS basis functions is readily available, which can be used to construct the above mentioned function $y$. 
Hence, applying the technique of functional-type a posteriori error estimation in a setting that relies only on the use of already available NURBS basis functions is greatly appealing.
The remainder of this paper is organized as follows. 
In Section~\ref{sec_prelims}, we define the model problem, and recall the definition and some important properties of B-spline and NURBS basis functions.
{In Section~\ref{sec_FuncEE}, we first recall functional-type a posteriori error estimates and known implementation issues. Then, we derive a quality criterion and the local error indicator.}
In Section~\ref{sec_effcomp}, we discuss a cost-efficient realization of the proposed error estimator using an illustrative numerical example. 
Further numerical examples are presented in Section~\ref{sec_numex}, and finally, conclusions are drawn in Section~\ref{sec_Conc}.

\section{Preliminaries}
\label{sec_prelims}

In order to fix notation and to provide an overview, we define the model problem and recall the definition and some aspects of isogeometric analysis in this section.

\subsection{Model Problem}

Let $\Omega \subset \dsR^2$ be an open, bounded and connected Lipschitz domain with boundary $\partial \Omega$. 
We shall consider the following model problem:

Find the scalar function $u:\overline{\Omega} \rightarrow \dsR$ such that\begin{equation}
\begin{array}{rcl@{\qquad}l}
- \dvg ( A \nabla u) &=& f & \text{in~} \Omega,\\
 u &=& u_D & \text{on~} \Gamma_D = \partial \Omega,
\end{array}
\label{e_defpde}
\end{equation}
where $A$, $f$ and $u_D$ are given data. We assume that $A$ is a symmetric positive definite matrix and has a positive inverse $A^{-1}$, and that there exist constants $c_1,c_2>0$ such that
\begin{eqnarray}
c_1 |\xi|^2 \leq A\xi \cdot \xi \leq c_2 |\xi|^2,\quad \forall \xi \in \dsR^2.
\label{e_Abounds}
\end{eqnarray}
Then, the norms
\begin{equation}
\| v \|_A^2 =  \int_\Omega A v \cdot v \dx, \quad
\| v \|_{\bar{A}}^2 = \int_\Omega A^{-1} v \cdot v \dx, \label{e_normaabar}
\end{equation}
are equivalent to the $L^2$-norm $\| v \|^2 = \int_\Omega v \cdot v \dx$.
The weak form of problem \eqref{e_defpde} can be written as follows:

Find $u \in V_g$, such that
\begin{eqnarray}
a(u,v) = f(v), \quad \forall v\in V_0,
\label{e_defvarp}
\end{eqnarray}
where $V_0 \subset H^1(\Omega)$ contains the functions which vanish on $\Gamma_D$, and $V_g \subset H^1(\Omega)$ contains the functions satisfying the Dirichlet boundary conditions $u=u_D$ on $\Gamma_D$.
We assume that the problem data $A$, $f$ and $u_D$ are given such that the bilinear form $a(\cdot,\cdot)$ is bounded, symmetric and positive definite, and that $f(\cdot)$ is a bounded linear functional. The energy norm of a function $v$ is given by $\| \nabla v \|_A = \sqrt{a(v,v)}$. 
Note that we have considered the Dirichlet problem only for the sake of simplicity. Functional-type error estimates can be easily generalized to problems with mixed boundary conditions, see, e.g., \cite{Lazarov2009_952,Repin08_book}.
We discretize the problem \eqref{e_defvarp} in the standard way by choosing a finite-dimensional manifold $V_h \subset V_g$ and looking for a \emph{discrete solution} $u_h \in V_h$. This leads to a linear system of equations of the form
\begin{eqnarray}
\underline{K}_h \underline{u}_h = \underline{f}_h,
\label{e_Khuhfh}
\end{eqnarray}
where $\underline{K}_h$ is the stiffness matrix induced by the bilinear form $a(\cdot,\cdot)$, $\underline{f}_h$ is the load vector, and $\underline{u}_h$ is the coefficient vector of the {discrete solution $u_h$}.

\subsection{B-Splines, NURBS and Isogeometric Analysis}
\label{sec_Bsplines}

We briefly recall the definition of B-spline basis functions and NURBS mappings. We only provide the basic definitions and properties relevant for the scope of this paper. 
For detailed discussions of NUR\-BS basis functions, geometry mappings and their properties, we refer to, e.g., \cite{Coott09_IGA,Cott07_4160,Hughes05_4135,Piegl1997} and the references therein. The following standard definitions and statements can also be found there.

Let $p$ be a non-negative \emph{degree} and let $s = (s_1,\ldots,s_{m})$ be a \emph{knot vector}
with $s_i \leq s_{i+1}$ for all $i$.
We consider only \emph{open knot vectors}, i.e., knot vectors $s$ where the multiplicity of a knot is at most $p$, except for the first and last knot which have multiplicity $p+1$.
For simplicity, we assume that $s_1 = 0$ and $s_{m}=1$, which can be easily achieved by a suitable scaling.
The $n = m-p-1$ univariate \emph{B-spline basis functions} $B_{i,p}^{s}: (0,1)\rightarrow \dsR$, $i=1,\ldots,n$, are defined recursively as follows:
\begin{eqnarray*}
B_{i,0}^{s}(\xi) &=&
	\left\{\begin{array}{c@{\quad}l@{\quad}l}
		1 & \mathrm{for} & s_i \leq \xi < s_{i+1} \\
		0 & \mathrm{else}
	\end{array}\right.\\
B_{i,p}^{s}(\xi) &=&
	\frac{\xi-s_{i}}{s_{i+p}-s_{i}} B_{i,p-1}^{s}(\xi)
	+\frac{s_{i+p+1}-\xi}{s_{i+p+1}-s_{i+1}} B_{i+1,p-1}^{s}(\xi).
\end{eqnarray*}
Whenever a zero denominator appears in the definition above, the corresponding function $B^s_{i,p}$ is zero, and the whole term is considered to be zero.
For open knot vectors, the first and last basis function are interpolatory at the first and the last knot, respectively.
The derivatives of B-spline basis functions are given by the following formula:
\begin{eqnarray*}
\partial_\xi B_{i,p}^s(\xi)
&=&
\frac{p}{s_{i+p}-s_i} B_{i,p-1}^s(\xi) - \frac{p}{s_{i+p+1}-s_{i+1}} B_{i+1,p-1}^s(\xi).
\end{eqnarray*}

B-spline basis functions of degree $p$ are, in general, globally $C^{p-1}$-continuous. 
In the presence of repeated knots, the continuity reduces according to the multiplicity, i.e., if a knot appears $k$ times, the continuity of a B-spline basis function of degree $p$ at that knot is $C^{p-k}$.

Let $\{B_{i,p}^s\}_{i=1}^{n_1}$ and $\{ B_{j,q}^t \}_{j=1}^{n_2}$ be two families of B-spline basis functions defined by the degrees $p$ and $q$, and the open knot vectors 
\[
s =(s_1,\ldots,s_{n_1+p+1}),\ 
t =(t_1,\ldots,t_{n_2+q+1}), 
\]
respectively.
We denote the set of all double-indices $(i,j)$ by 
\[
 \caI_R = \{(i,j):\ i\in\{1,\ldots,n_1\}, j\in\{1,\ldots,n_2\}\}.
\]
Let $w_{(i,j)}$, $(i,j)\in\caI_R$, be positive \emph{weights}.
The \emph{bivariate NURBS basis functions} $R_{(i,j)}(\xi_1,\xi_2)$, $(i,j)\in\caI_R$ are defined as follows:
\begin{eqnarray*}
  R_{(i,j)}(\xi_1,\xi_2) &=&
  \frac{ { w_{(i,j)} \cd  B_{i,p}^s(\xi_1) \cd B^t_{j,q}(\xi_2) } }{ \sum_{(k,\ell)\in\caI_R } w_{(k,\ell)} B_{k,p}^s(\xi_1) \cd B^t_{\ell,q}(\xi_2) }.
\end{eqnarray*}
The continuity of the B-spline basis functions is inherited by the NURBS basis functions. Note that B-splines can be seen as a special case of NURBS with all weights being equal to one. Hence, we will not distinguish between these two and we will only use the term \emph{NURBS} in the remainder of the paper.
The set of functions 
\[
\hat{V}_h = \spn \{ R_{(i,j)},\ (i,j)\in\caI_R \},
\]
associated with the \emph{parameter domain} $\hat{\Omega} = (0,1)^2$, is uniquely determined by the degrees $p$ and $q$, the knot vectors $s$ and $t$, and the weights $w$.
To reflect the associated polynomial degrees in respective dimensions, we will also use the notation $\caS^{p,q}_{h}$ for $\hat{V}_h$, which denotes the NURBS function of degree $p$ and $C^{p-1}$-continuity in the first coordinate, degree $q$ and $C^{q-1}$-continuity in the second coordinate, and where the parameter $h$ is the characteristic cell size (non-vanishing knot-span) of the mesh for $\hat{V}_h$.
Given the set of functions $\hat{V}_h$ and a \emph{control net} of \emph{control points} $P_{(i,j)} \in \dsR^2$, where $(i,j)\in \caI_R$, the two-dimensional \emph{NURBS-surface} $G:\hat{\Omega}\rightarrow \Omega$ is defined by
\begin{equation}
  G(\xi_1,\xi_2) =
  \sum_{(i,j)\in\caI_R} R_{(i,j)}(\xi_1,\xi_2)\cd P_{(i,j)}.
  \label{glg_NURBS_Surf}
\end{equation}
We refer to $\Omega = G(\hat{\Omega})$ as the \emph{physical domain}.
We assume that the geometry mapping is continuous and bijective (i.e., not self-penetrating), which are natural assumptions for CAD-applications.

In isogeometric analysis, the isoparametric principle is applied by using the same basis functions for the discrete solution $u_h$ which are used for representing the geometry. 
For detailed discussion, we refer the reader to, e.g., \cite{Coott09_IGA,Cott07_4160,Hughes05_4135}.
The discrete solution $u_h$ on the physical domain $\Omega$ is represented as follows:
\begin{eqnarray}
u_h(x) = \sum_{(i,j)\in\caI_R} u_{(i,j)}\ \big({R}_{(i,j)} \circ G^{-1}\big)(x),
\end{eqnarray}
where $u_{(i,j)}\in \dsR$ are real-valued coefficients which form the coefficient vector $\underline{u}_h$. The discrete functions space is thus defined by
\[
V_h = \spn \{ R_{(i,j)} \circ G^{-1},\ (i,j)\in \caI_R\}.
\]

The initial mesh, and thereby the basis functions on this initial mesh, are assumed to be given via the geometry representation of the computational domain, i.e., the initial discretization is already determined by the problem domain.
The exact representation of the geometry on the initial (coarsest) level is preserved in the process of mesh refinement.

As mentioned in the introduction, the straightforward definition of NURBS basis functions, leads to a tensor-product structure of the discretization, which is the focus of this paper. Nevertheless, the error estimator presented herein is also applicable to local refinement techniques (e.g., T-splines, THB-splines, PHT-splines, LR-splines, see Section~\ref{sec_intro}) since it is derived purely on functional grounds.

\section{Functional-type a Posteriori Error Estimates}
\label{sec_FuncEE}

{In the first two parts of this section, we will discuss the well-known theoretical upper bound for the error in the energy norm (see, e.g., \cite{Repin97_201,Repin99_4311,Repin00_481,Repin08_book}), and we recall how to minimize this upper bound in order to get a sharp error estimate (see, e.g., \cite{Kraus11_1175,Lazarov2009_952}). Thereafter, in Section~\ref{sec_qualcrit}, we will derive a quality criterion from the discussed theory.} We will comment on the realization in the isogeometric context in Section~\ref{sec_effcomp}.

\subsection{Guaranteed Upper Bound for the Error}
\label{sec:Majorant}

The starting point for the proposed method is the following main result, which gives an upper bound for the error in the energy norm. It can be found, e.g., in \cite{Repin99_4311,Repin00_481,Repin08_book}. 
\begin{theorem}\label{thm_est1}
Let $C_\Omega$ be the constant in the Friedrich's type inequality $\| v \| \leq C_{\Omega} \| \nabla v \|_A,\ \forall v \in V_0$.
Let $u$ be the exact solution of the problem \eqref{e_defvarp}, and let $u_h\in V_h$ be an approximate solution. 
Then, the following estimate holds:
\begin{equation}
\| \nabla  u - \nabla u_h \|_A
\leq
\|  A \nabla u_h - y \|_{\bar{A}} + C_\Omega \| \dvg y + f \|,
\label{e_DefMoplusRaw}
\end{equation}
where $y$ is an arbitrary vector-valued function in $H(\Omega,\dvg)$, and the norms are as defined in \eqref{e_normaabar}. \end{theorem}%

\noindent The constant $C_\Omega$ depends only on the domain $\Omega$ and the coefficient matrix $A$ (but not on the underlying mesh), see, e.g., \cite{Lazarov2009_952,Repin08_book}.
Note that $C_\Omega$ can be computed either numerically or, if one can find a domain $\Omega_\square \supset \Omega$, where $\Omega_\square$ is a square domain with side-length $\ell$, then $C_\Omega \leq c_2 \tfrac{\ell}{\pi \sqrt{d}}$, where $d$ is the dimension and $c_2$ is the constant in \eqref{e_Abounds}.

Note that, if we choose $y$ via the (unknown) exact solution $y=A\nabla u$, both sides of \eqref{e_DefMoplusRaw} coincide. Hence, the estimate is sharp in the sense that, for any fixed $u_h$, we can find a function $y$ such that the upper bound is as close to the exact error as desired.
The estimate given in Theorem~\ref{thm_est1} is a guaranteed and fully computable upper bound for any conforming approximation $u_h \in V_{g}$.
In the following, we describe some approaches to construct the function $y$ and discuss their relative merits. For this reason, we consider a numerical example, referred to as \emph{Example~\ref{ex:sin_6pix_3piy}} in the remainder, whose solution is a smoothly varying function in both directions.
\begin{example}\label{ex:sin_6pix_3piy}
\textbf{Sinus function in a unit square:} In this numerical example, the computational domain is the unit square $\Omega = (0,1)^2$ and $u_h \in \caS^{2,2}_{h}$, i.e., a piecewise quadratic function in both directions. The coefficient matrix is the identity matrix, i.e., $A=I$, and the exact solution is given by
\[ u = \sin(6\pi x) \sin(3\pi y).\]
The right-hand-side $f$ and the (homogeneous) boundary conditions $u_D$ are determined by the prescribed exact solution $u$.
\end{example}

\subsubsection{Post-processing of \texorpdfstring{$u_{h}$}{}}
\label{sec:PostProc}

It is possible to obtain good error \emph{indicators} by constructing a function $y$ by some post-processing of the discrete solution $u_h$, see \cite{Lazarov2009_952,Repin08_book} and the references therein. Since $u_{h} \in C^{p-1}$, we have $\nabla u_{h} \in (C^{p-2})^{2} \subset H(\Omega,\dvg)$ for $p \ge 2$. Choosing $y = \nabla u_{h}$ will thus result in
\begin{align}
\Vert \nabla  u - \nabla u_h \Vert \leq C_{\Omega} \Vert \Delta u_{h} + f \Vert.
\label{eq:yh_graduh}
\end{align}
Once we have calculated $\eta_Q := \Vert \Delta u_{h} + f \Vert_{Q}$ for each cell $Q$ of the mesh, we can compare the local errors and choose a criterion for selecting cells which will be marked for further refinement. Typically, one chooses a threshold $\Theta$ and marks all cells $Q$ for refinement, where the local error is above this threshold. There are several possibilities for determining $\Theta$, e.g., the bulk-criterion proposed in \cite{Doerfler96_1106}. For simplicity, we choose a percentage $\psi$ and mark a cell $Q$ for refinement, if
\begin{equation}
\eta_Q > \Theta ,\text{~where~} \Theta = (100-\psi)\text{-percentile~of~}\{\eta_Q\}_Q.
\label{e_0228a}
\end{equation}
The $\alpha$-percentile of a set $\caA=\{a_1,\ldots,a_\nu\}$ denotes the value $\bar{a}$ below which $\alpha$ percent of all values $a_i$ fall. For example, if we choose $\psi = 20\%$ in \eqref{e_0228a}, then $\Theta$ is chosen such that $n_Q>\Theta$ holds for 20\% of all cells $Q$.
\setlength{\picw}{3cm}
\begin{figure}[!ht]\centering
\subfigure[$16\times 16$]{\includegraphics[height=\picw,width=\picw]{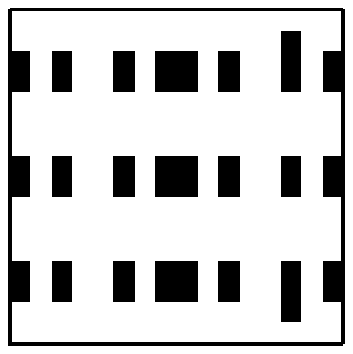}} \qquad
\subfigure[$32\times 32$]{\includegraphics[height=\picw,width=\picw]{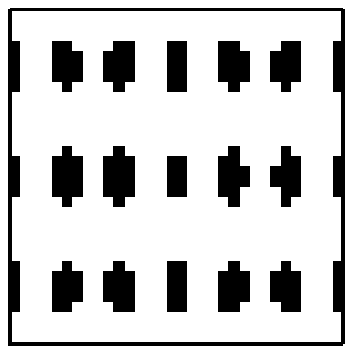}} \qquad
\subfigure[$64\times 64$]{\includegraphics[height=\picw,width=\picw]{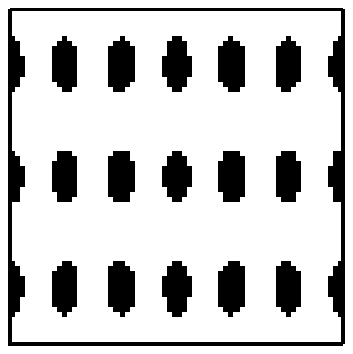}} \qquad
\subfigure[$128\times 128$]{\includegraphics[height=\picw,width=\picw]{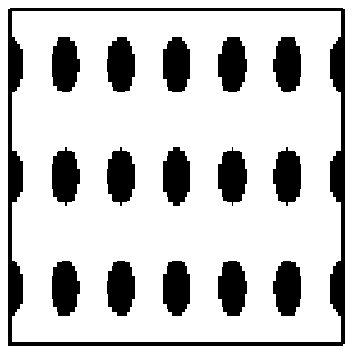}}
\caption{Cells marked by exact error with $\psi = 20\%$ in Example~\ref{ex:sin_6pix_3piy}, $\hat{V}_{h} = \caS^{2,2}_{h}$.
\label{fig_nexSU_Ex}}
\end{figure}
\setlength{\picw}{3cm}
\begin{figure}[!ht]\centering
\subfigure[$16\times 16$]{\includegraphics[height=\picw,width=\picw]{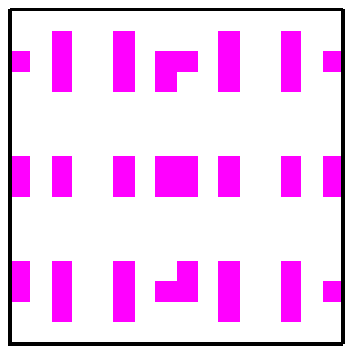}} \qquad
\subfigure[$32\times 32$]{\includegraphics[height=\picw,width=\picw]{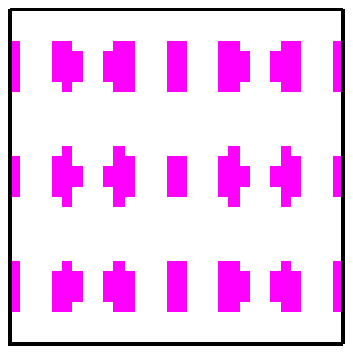}} \qquad
\subfigure[$64\times 64$]{\includegraphics[height=\picw,width=\picw]{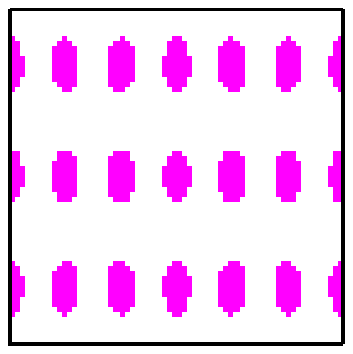}} \qquad
\subfigure[$128\times 128$]{\includegraphics[height=\picw,width=\picw]{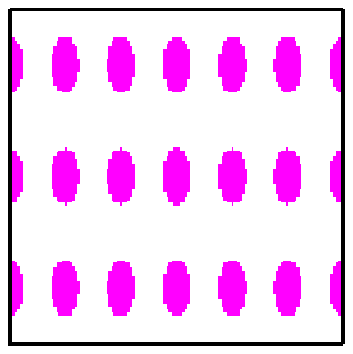}}
\caption{Cells marked by error estimator with $\psi = 20\%$ in Example~\ref{ex:sin_6pix_3piy}, $\hat{V}_{h} = \caS^{2,2}_{h}, y_{h} = \nabla u_{h}$.
\label{fig_nexSU_graduh}}
\end{figure}
\begin{figure}[!ht]\centering
\includegraphics[scale=.5]{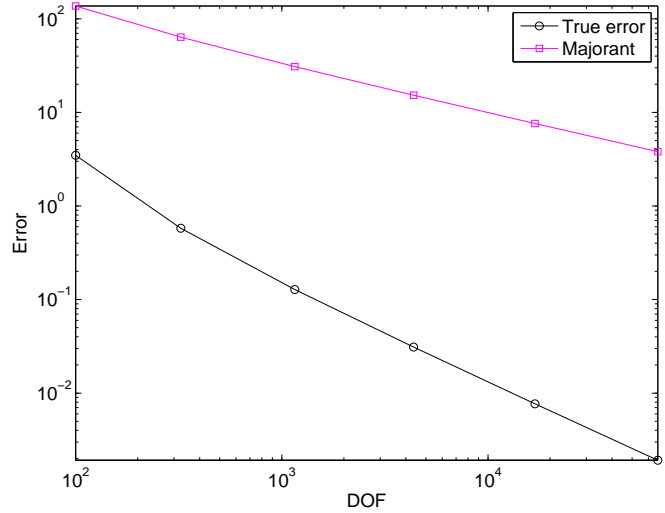}
\caption{Convergence of exact error and the majorant \eqref{eq:yh_graduh} for Example~\ref{ex:sin_6pix_3piy}.\label{fig_Conv_EstErr}}
\end{figure}
To show the efficiency of the estimator \eqref{eq:yh_graduh}, in Figure~\ref{fig_nexSU_Ex}, we present the cells marked for refinement by the exact error. The cells marked for refinement by the majorant given in \eqref{eq:yh_graduh} are presented in Figure~\ref{fig_nexSU_graduh}.
We see that starting from the mesh $32 \times 32$, the majorant is able to nicely capture the refinement pattern of exact error. However, from a closer look at the convergence of the exact error and the majorant, see Figure~\ref{fig_Conv_EstErr}, we find that though such an estimate is a guaranteed upper bound and very cheap to compute, it over-estimates the exact error, and its convergence is slower than the exact error (due to a lack of proper scaling, different operators acting on $u_{h}$ on both sides).
\footnote{We also studied a patch-wise interpolation approach. Unfortunately, this approach is neither a cheap one (to compute $y$) nor does it result in desired efficiency indices in the proximity of $1$, and therefore, we do not present its results.}
\subsubsection{Global minimization}
\label{sec:Minimization}

In order to obtain a sharp estimate (and not just an indicator), therefore, one has to find a function $y$ which minimizes the right-hand-side of \eqref{e_DefMoplusRaw}. 
For minimizing the estimate \eqref{e_DefMoplusRaw} numerically, we first rewrite the estimate in the following form
\begin{equation}
\| \nabla  u - \nabla u_h \|_A^2 
\leq 
(1+\beta) \|  A \nabla u_h - y \|_{\bar{A}}^2 + (1+\tfrac{1}{\beta}) C_\Omega^2 \| \dvg y + f \|^2
\ =:\ M_\oplus^2(y,\beta),
\label{e_DefMoplus}
\end{equation}
where $\beta > 0$ is a free parameter \cite{Lazarov2009_952,Repin08_book}. Note that the upper bound in \eqref{e_DefMoplus} holds true for \emph{any} fixed $y\in H(\Omega,\dvg)$ and  $\beta>0$.
Hereinafter, for simplicity, we will refer to $M_\oplus^2(y,\beta)$ as the \emph{majorant}. Introducing 
\begin{equation}
\begin{array}{r@{\ =\ }l@{\qquad}r@{\ =\ }l}
a_1 & 1+\beta, &
a_2 & (1+\tfrac{1}{\beta})C_\Omega^2,\\
B_1 & \displaystyle \| A\nabla u_h - y \|_{\bar{A}}^2, &
B_2 & \displaystyle \| \dvg y + f\|^2,
\end{array}\label{e_defaB}
\end{equation}
we can briefly write the majorant as
\begin{equation}
M_\oplus^2(y,\beta) = a_1 B_1 + a_2 B_2.
\end{equation}
The \emph{efficiency index}, defined by
\begin{align}
I_{\text{eff}} = \frac{M_\oplus(y,\beta)}{\| \nabla u - \nabla u_h \|_A},
\label{eq:Def_Ieff}
\end{align}
indicates how close the calculated majorant is to the exact error. The closer $I_{\text{eff}}$ is to 1, the better the estimate.
Therefore, obtaining a \emph{sharp} estimate requires to find $y\in H(\Omega,\dvg)$ and $\beta >0$ as solutions to the global minimization problem 
\begin{equation}
\min_{y\in H(\Omega,\dvg),\ \beta >0} M_\oplus^2(y,\beta).
\end{equation}
The technique for finding such minimizing parameters $y$ and $\beta$ will be discussed 
in Sections~\ref{sec_minmaj} and~\ref{sec_NumRealIgA}.
Before proceeding further, we give the following Lemma~\ref{lem_Moconv}, which can be found in \cite[Prop.\ 3.10]{Repin08_book}. 
It provides an analytical result on the sharpness of the bound $M_\oplus^2(y,\beta)$. For later reference, we also sketch the proof.
\begin{definition}
A sequence of finite-dimensional subspaces $\{ Y_j \}_{j=1}^\infty$ 
of a Banach-space $Y$ is called \emph{limit dense in $Y$}, 
if for any $\varepsilon >0$ and any $v\in Y$, 
there exists an index $j_\varepsilon$, such 
that $\inf_{p_k\in Y_k} \|p_k - v \|_{Y} < \varepsilon$ for all $k>j_\varepsilon$ .
\end{definition}

\begin{lemma}\label{lem_Moconv}
Let the spaces $\{Y_j\}_{j = 1}^\infty$ be limit dense in $H(\Omega,\dvg)$.
Then
\[
\lim_{j\to\infty}\ \inf_{y_j\in Y_j, \beta > 0}\ M_\oplus^2(y_j,\beta) = \| \nabla u - \nabla u_h \|_A^2.
\]
\end{lemma}

\begin{proof}
Recall that the $H(\Omega,\dvg)$-norm $\|\cdot\|_{\dvg}$ is defined by $\|v\|_{\dvg}^2 = \|v\|^2 + \|\dvg v\|^2$.
Let $\varepsilon > 0 $ be arbitrarily small, but fixed. Let $j_\varepsilon$ be the index such that, for all $k > j_\varepsilon$, there exists a $p_k \in Y_k$ with $\|A \nabla u -p_k\|_{\dvg} < \varepsilon$.
Then,
\begin{equation}
\inf_{y_j\in Y_j, \beta > 0}\ M_\oplus^2(y_j,\beta)
\leq
\ M_\oplus^2(p_k,\varepsilon)
=
(1+\varepsilon) {\| A \nabla u_h - p_k \|_{\bar{A}}^2}
+ ( 1 + \tfrac{1}{\varepsilon}) C_\Omega^2 { \| f + \dvg p_k \|^2 }.
\label{e_0212c}
\end{equation}
Since $\| A v \|_{\bar{A}} = \| v\|_A$, we can write
\begin{eqnarray*}
\| A \nabla u_h - p_k \|_{\bar{A}}
&\leq &
\| A \nabla u_h - A \nabla u \|_{\bar{A}} + \| A \nabla u - p_k  \|_{\bar{A}} \\
&=&
\| \nabla u_h - \nabla u \|_{A} + \| A \nabla u - p_k  \|_{\bar{A}}.
\end{eqnarray*}
The norm $\| \cdot \|_{\bar{A}}$ is equivalent to the $L^2$-norm, so there exists a constant $c_A$, such that the second term in the right-hand side can be bounded by
\[
\| A \nabla u - p_k  \|_{\bar{A}}
\leq c_A \| A \nabla u - p_k  \|
\leq c_A \| A \nabla u - p_k  \|_{\dvg}
\leq c_A \varepsilon.
\]
Hence, we obtain the following estimate for the first term in \eqref{e_0212c}:
\begin{equation}
{\| A \nabla u_h - p_k \|_{\bar{A}}} \leq
\| \nabla u - \nabla u_h \|_A + \caO(\varepsilon).
\label{e_0212a}
\end{equation}
Since $f = -\dvg A \nabla u$, we can bound the second term in \eqref{e_0212c} as follows:
\begin{equation}
\| \dvg p_k + f \|
= 
\| \dvg p_k - \dvg A \nabla u \|
\leq 
\| p_k - A \nabla u \|_{\dvg} \ 
\leq \ \varepsilon.
\label{e_0212b}
\end{equation}
With \eqref{e_0212a} and \eqref{e_0212b}, we can rewrite \eqref{e_0212c} as
\begin{eqnarray*}
M_\oplus^2(p_k,\varepsilon)
\leq (1+\varepsilon) ( \| \nabla u - \nabla u_h \|_A^2 + \caO(\varepsilon) ) 
+ (1+\tfrac{1}{\varepsilon}) C_\Omega^2 \varepsilon^2
= \| \nabla u - \nabla u_h \|_A^2 + \caO(\varepsilon).
\end{eqnarray*}
Hence, the bound $M_\oplus^2(p_k,\varepsilon) \rightarrow \| \nabla u - \nabla u_h \|_A^2 $ as $\varepsilon \rightarrow 0$.
\end{proof}

\subsection{Steps Involved in Minimizing \texorpdfstring{$M_\oplus^2(y,\beta)$}{}}
\label{sec_minmaj}

As mentioned above, we need to find parameters $y$ and $\beta$ which minimize the majorant.
To do this, we apply an interleaved iteration process in which we alternately fix one of the variables and minimize with respect to the other. This process, which we summarize in the following, has been described, e.g., in \cite{Kraus11_1175,Lazarov2009_952}.

\begin{description}
\item[Step 1] Minimization with respect to $y$:
Assume that $\beta >0$ is given and fixed, either by an initial guess or as a result of Step 2 below. We view the majorant $M_\oplus^2(y)$ as a quadratic function of $y$ and calculate its Gateaux-derivative $M_\oplus^2(y)'$ with respect to $y$ in direction $\tilde{y}$. Setting $M_\oplus^2(y)' = 0$, we obtain
\begin{eqnarray}
a_1 \int_\Omega A^{-1} y \cdot \tilde{y} \dx + a_2 \int_\Omega \dvg y\ \dvg \tilde{y} \dx
&=&
a_1 \int_\Omega \nabla u_h \cdot \tilde{y} \dx - a_2 \int_\Omega f\ \dvg \tilde{y} \dx,
\label{e_minwrty}
\end{eqnarray}
where $a_1 = 1+\beta$ and $a_2 = (1+\tfrac{1}{\beta})C_\Omega^2$, as defined in \eqref{e_defaB}.
In order to solve \eqref{e_minwrty}, we choose a finite-dimensional subspace $Y_h\subset H(\Omega,\dvg)$ and search for a solution $y_h \in Y_h$. Testing in all directions $\tilde{y}\in Y_h$ leads to a linear system of equations which we write as
\begin{eqnarray}
\underline{L}_h \underline{y}_h = \underline{r}_h.
\label{e_Lhyhrh}
\end{eqnarray}
Here, $\underline{L}_h$ and $\underline{r}_h$ are the matrix and the vector induced by the left hand side and the right hand side of equation \eqref{e_minwrty}, respectively.
By solving \eqref{e_Lhyhrh}, we obtain the coefficient vector $\underline{y}_h$ for the discrete function $y_h$ minimizing $M_\oplus^2(y)$ in $Y_h\subset H(\Omega,\dvg)$.
Note that this process requires non-negligible cost as we need to assemble $\underline{L}_h$ and $\underline{r}_h$ and solve the system \eqref{e_Lhyhrh}.

\item[Step 2] Minimization with respect to $\beta$:
Assume that $y_h$ is given from Step 1. By direct calculation, we see that $M_\oplus^2(\beta)$ is minimized with respect to $\beta$ by setting
\begin{eqnarray}
\beta &=&  C_{\Omega} \sqrt{\frac{B_2}{B_1}},
\label{e_minwrtb}
\end{eqnarray}
where $B_1$ and $B_2$ are as defined in \eqref{e_defaB}.
Note that the evaluation of $B_1$ and $B_2$ (and thus $\beta$) requires only the evaluation of integrals, and thus involves negligible cost.
\end{description}

Steps~1 and 2 are repeated iteratively. 
We will refer to one loop of applying Step~1 and Step~2 as one \emph{interleaved iteration}.
Once we have computed minimizers $y_h$ and $\beta$, the computation of the majorant $M_\oplus^2(y_h,\beta)$ is straight-forward as it requires only the evaluation of the integrals.

Note that the matrix $\underline{L}_h$ can be written as
\begin{equation}
\underline{L}_h = a_1 \underline{L}_h^1 + a_2 \underline{L}_h^2,
\label{e_0319a}
\end{equation}
where $\underline{L}_h^1$ and $\underline{L}_h^2$ correspond to the terms $\int_\Omega A^{-1} y \cdot \tilde{y} \dx$ and $\int_\Omega \dvg y\ \dvg \tilde{y} \dx$ in \eqref{e_minwrty}, respectively.
Since the matrices $\underline{L}_h^1$ and $\underline{L}_h^2$ in \eqref{e_0319a} do not change in the interleaved iteration process, they need to be assembled only once. Analogously to \eqref{e_0319a}, we can write $\underline{r}_h$ as
\begin{align}
\underline{r}_h = a_1 \underline{r}_h^1 - a_2 \underline{r}_h^2,
\label{eq:vec_rh}
\end{align}
where $\underline{r}_h^1$ and $\underline{r}_h^2$ correspond to the terms $\int_\Omega \nabla u_h \cdot \tilde{y}_h \dx$ and $\int_\Omega f \ \dvg \tilde{y} \dx$ in \eqref{e_minwrty}, respectively.
The terms $\underline{r}_h^1$ and $\underline{r}_h^2$ also need to be assembled only once since they also do not change in the interleaved iteration process.
The full matrix $\underline{L}_h$ and vector $\underline{r}_h$, however, do change in each iteration, because of the change in $\beta$ and $y_h$. Based on past numerical studies, see, e.g., \cite{Kraus11_1175,Lazarov2009_952}, and the results presented in Sections~\ref{sec_effcomp} and \ref{sec_numex}, it has been found that for linear problems,
one or two such interleaved iterations are enough for obtaining a sufficiently accurate result.

To recapitulate, we summarize the steps for computing the majorant in Algorithm~\ref{algo:Mplus}.
\begin{algorithm}[!ht]
\caption{Computation of the majorant $M_\oplus$}
\label{algo:Mplus}
\begin{algorithmic}
\Require $u_h$, $f$, $C_\Omega$, $Y_h$
\Ensure $M_\oplus$
\State $\beta:=$ initial guess
\State Assemble and store $\underline{L}_h^1$, $\underline{L}_h^2$, $\underline{r}_h^1$, $\underline{r}_h^2$
\While{convergence is not achieved or maximum number of interleaved iterations is not reached}
\State $\underline{L}_h := (1+\beta) \underline{L}_h^1 + (1+\tfrac{1}{\beta}) C_\Omega^2 	\underline{L}_h^2$
\State $\underline{r}_h := (1+\beta) \underline{r}_h^1 - (1+\tfrac{1}{\beta}) C_\Omega^2 \underline{r}_h^2$
\State Solve $\underline{L}_h \underline{y}_h = \underline{r}_h$ for $\underline{y}_h$
\State $B_1 := \| A \nabla u_h - y_h \|^2_{\bar{A}}$
\State $B_2 := \| \dvg y_h + f \|^2$
\State $\beta := C_\Omega \sqrt{ B_2 / B_1 }$
\EndWhile
\State $M_\oplus(y,\beta) := \sqrt{(1+\beta)B_1 + (1+\tfrac{1}{\beta})C_\Omega^2 B_2}$
\end{algorithmic}
\end{algorithm}

\begin{remark}\label{rem_yglobal}
Note that the space $H(\Omega,\dvg)$, where the auxiliary quantity $y$ is sought, is a global space, and for a general complicated problem, it is not immediately clear how to locally compute $y$ without global effect. That being said, a local version of our estimator can be devised for specific problems and data (like equilibration of flux approach), however, that will restrict its generality, which is not very appealing to us. Therefore, in the remainder of the paper, we will focus on computing the majorant from the global minimization problem.
\end{remark}

\subsection{Quality Indicator and Local Error Indicator}
\label{sec_qualcrit}

So far, we have defined the majorant and discussed how we minimize (numerically) the majorant over $Y_h$.
Another important question, especially in the light of adaptive, local refinement, is whether a calculated majorant does correctly capture the error distribution.
From the proof of Lemma~\ref{lem_Moconv}, we recall the following observation:
\begin{equation}
a_1 B_1 \to \| \nabla u - \nabla u_h \|_A^2
\text{~and~}
a_2 B_2 \to 0,
\text{~as~} y_h \in H(\Omega,\dvg) \to A\nabla u.
\label{e_1stCompConv}
\end{equation}
From this, we deduce the following quality indicator.

\begin{proposition}
The distribution of the exact error is captured correctly, if 
\begin{equation}
a_1 B_1 > C_\oplus \ a_2 B_2 \label{e_reli}
\end{equation}
with some constant $C_\oplus > 1$.
\end{proposition} 
\noindent This criterion is easy to check, since the terms appearing in \eqref{e_reli} are evaluated in the process of minimizing $M_\oplus^2(y,\beta)$. It was found from extensive numerical studies (see examples presented in Sections~\ref{sec_effcomp} and \ref{sec_numex}) that an accurate distribution of the error is obtained for $C_\oplus \ge 5$.

\begin{remark}\label{rem_Cplus}
For the choice of $C_{\oplus} \ge 5$, we have $a_2 B_2 < a_1 B_1 /5$, and therefore, $\Vert \nabla u - \nabla u_h \Vert_A \le \sqrt{1.2~ a_1 B_1}$. One can see from all the tables in Sections~\ref{sec_effcomp} and \ref{sec_numex}, that whenever this criterion is satisfied, we have $I_{\text{eff}} \le 1.2$ (the ratio of $\sqrt{a_1 B_1}/\Vert \nabla u - \nabla u_h \Vert_A$ appears to be of the same magnitude as $\sqrt{1 + 1/C_{\oplus}}$. Note that this criterion does not require $a_2 B_2$ to be close to zero, but just less than $a_1 B_1 /5$. Since these approximations (of the original problem and the auxiliary problem in $H(\Omega, \dvg)$) are monotonically convergent, the approximation at any level will only improve at the next refinement level, and this is why the results get better for any further refinement. Clearly, all the terms are fully computable, and thus, usable in an algorithm.
\end{remark}

{We define the local error indicator $\eta_Q$ on a cell $Q$ as the restriction of the first component of the majorant to the cell $Q$, i.e., by
\begin{align}
\eta_Q^2(y_h) = \int_Q (\nabla u_h - A^{-1}y_h ) (A\nabla u_h - y_h )\dx.
\label{e_errest}
\end{align}
The factor $(1+\beta)$ is omitted, since this scalar factor is the same for all cells of the domain. As remarked in the observation \eqref{e_1stCompConv}, the first component will converge to the exact error, thus providing a good indicator for the error distribution. A more detailed discussion of this indicator can be found in \cite[Sec.~3.6.4]{Repin08_book}.}
For refinement based on $\eta_{Q}$, we again use the criterion \eqref{e_0228a}.
\section{Efficiency and Computational Cost of the Proposed Estimator in the Isogeometric Context}
\label{sec_effcomp}

We now discuss the efficiency and the computational cost of the proposed estimator based on the global minimization steps presented in Section~\ref{sec_minmaj}. Through out this Section, we again consider Example~\ref{ex:sin_6pix_3piy} from Section~\ref{sec:Majorant}.
All the computations for this example and the examples presented in Section~\ref{sec_numex} are performed in {MATLAB}$^\circledR$ on an HP workstation Z420 with Intel Xeon CPU E5-1650, 3.2 GHz, 12 Cores and 16 GB RAM, and the linear systems \eqref{e_Khuhfh} and \eqref{e_Lhyhrh} are solved using the in-built direct solver.
The right-hand-side $f$ and the boundary conditions $u_D$ are determined by the prescribed exact solution $u$.
We study the efficiency of the majorant based on \emph{straight forward} computational procedure, as discussed in Section~\ref{sec_proc0}, and based on \emph{cost-efficient} procedure, as discussed in Section~\ref{sec_NumRealIgA}, which coarsens the mesh and increases the polynomial degree simultaneously. This alternative cost-efficient procedure will then be used in Section~\ref{sec_numex} for further numerical examples. In all the numerical results of Example~\ref{ex:sin_6pix_3piy} in this Section, the initial guess for $\beta$ is $0.01$.
In the tables, we indicate the mesh-size by the number of interior knot spans of the knot vectors $s$ and $t$, respectively. By this, we mean the number of knot spans without counting the vanishing knot spans at the beginning and the end of the open knot vectors. For example, if 
\begin{eqnarray*} 
s&=&(0,0,0,0.25,0.5,0.75,1,1,1)\\
t&=&{(0,0,0,0,0.5,0.5,1,1,1,1)},
\end{eqnarray*}
then the mesh-size is $4 \times {3}${, since the empty knot span $(0.5,0.5)$ in $t$ is also counted as an interior knot span}.
We compare the timings for assembling and for solving the linear systems \eqref{e_Khuhfh} and \eqref{e_Lhyhrh}, as well as the total time for assembling and solving. In the presented tables, these timings are shown in the columns labeled assembling-time, solving-time, and sum, respectively. The label \emph{pde} indicates that the column corresponds to solving the partial differential equation \eqref{e_Khuhfh}, i.e., to assembling $\underline{K}_h$ and solving \eqref{e_Khuhfh} for $\underline{u}_h$. The label \emph{est} indicates that the timings correspond to the estimator, i.e, assembling $\underline{L}_h$ and solving \eqref{e_Lhyhrh} for~ $\underline{y}_h$. {In the column labeled $\frac{\emph{est}}{\emph{pde}}$, we present the ratio of these timings. Note that these ratios were computed \emph{before} rounding the numbers, i.e., taking the ratios of the reported numbers may result in slightly different values.}
{The computed efficiency indices $I_{\text{eff}}$ (see \eqref{eq:Def_Ieff}) are presented in tables.} In order to check the quality criterion discussed in Section~\ref{sec_qualcrit}, we present the values of $a_1B_1$ and $a_2 B_2$ and see whether the inequality \eqref{e_reli} is fulfilled or not.
To indicate the quality of the error distribution captured by the majorant, we plot which cells are marked for refinement based on the exact local error and the criterion \eqref{e_0228a} {(plotted in black)}, and
compare this to the refinement marking based on the criterion \eqref{e_0228a} applied to the computed error estimate {(plotted in magenta)}.

\subsection{Straightforward Procedure}\label{sec_proc0}

\addtocounter{case}{-1}
\begin{case}\label{ex:case_0}
\textbf{(Straightforward Procedure)} For the first choice for $\hat{Y}_h$, we use the same mesh as for $\hat{V}_h$, and choose 
\begin{equation}
\hat{Y}_h = \caS^{p+1,p}_{h} \otimes \caS^{p,p+1}_{h}.
\label{e_Yh0}
\end{equation}
The function space $Y_h$ is then defined by the well known Piola transformation \cite{BoffiBF-13}.
\end{case}

We consider the same setting as presented in Example~\ref{ex:sin_6pix_3piy} in Section~\ref{sec:Majorant}. In Table~\ref{nexSUeff_0}, we present the computed efficiency indices obtained with this choice of $Y_h$, which show that upper bound approaches $1$ (representing exact error) as the mesh is refined. The dashed line in Table~\ref{nexSUeff_0} indicates that the criterion \eqref{e_reli} is fulfilled with $C_\oplus = 5$ (actually $4.94$) starting from the mesh $64 \times 64$.
The cells marked by the error estimator are shown in Figure~\ref{fig_nexSU_0}. When comparing these plots {to those presented in Figure~\ref{fig_nexSU_Ex}}, we see that the error distribution is captured accurately starting from the mesh $32\times 32$.
%

%
The timings presented in Table~\ref{nexSUtime_0}, however, show that the computation of the error estimate is costlier (about $4.5$ times) than assembling and solving the original problem. This is not surprising, since, when $N_u$ denotes the number of degrees of freedom (DOF) of $u_h$, the number of DOF of $y_h$, which is vector-valued, is asymptotically $2N_u$. This results in higher assembly time and the solution time for the linear system (where a direct solver is used). Clearly, this straightforward approach is not cost-efficient. In the next section, therefore, we discuss some cost-efficient approaches for computing $y_{h}$.
\setlength{\tabcolsep}{3pt}
\begin{table}[!ht]\begin{center}
\begin{tabular}{|c|r|rr|}\hline
mesh-size & \multicolumn{1}{|c|}{$I_{\text{eff}}$}& \multicolumn{1}{|c}{$a_1 B_1$}
& \multicolumn{1}{c|}{$a_2 B_2$}
\\
\hline
\hline
$8 \times 8$  & { 3.43 } & 2.62e+01 & 1.17e+02 \\
$16 \times 16$  & { 1.92 } & 6.07e-01 & 6.19e-01 \\
$32 \times 32$  & { 1.41 } & 2.29e-02 & 9.71e-03 \\ \hdashline
$64 \times 64$  & { 1.20 } & 1.15e-03 & 2.33e-04 \\
$128 \times 128$  & { 1.10 } & 6.51e-05 & 6.54e-06 \\
$256 \times 256$  & { 1.05 } & 3.87e-06 & 1.95e-07 \\
$512 \times 512$  & { 1.03 } & 2.36e-07 & 5.94e-09 \\
\hline
\end{tabular}
\caption{Efficiency index and components of the majorant in Example~\ref{ex:sin_6pix_3piy}, Case~\ref{ex:case_0}, $\hat{V}_{h} = \caS^{2,2}_{h}, \hat{Y}_h = \caS^{3,2}_{h} \otimes \caS^{2,3}_{h}$.\label{nexSUeff_0}}
\end{center}\end{table}

\setlength{\picw}{3cm}
\begin{figure}[!ht]\centering
\subfigure[$16\times 16$]{\includegraphics[height=\picw,width=\picw]{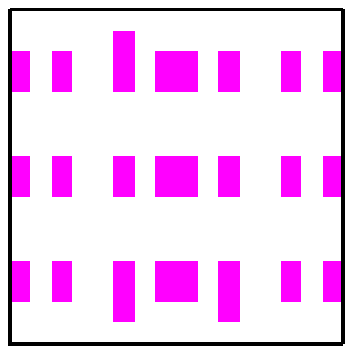}} \qquad
\subfigure[$32\times 32$]{\includegraphics[height=\picw,width=\picw]{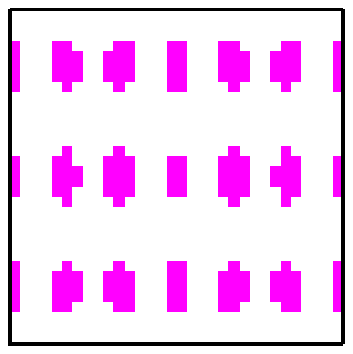}} \qquad
\subfigure[$64\times 64$]{\includegraphics[height=\picw,width=\picw]{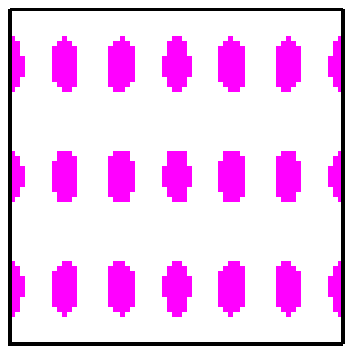}} \qquad
\subfigure[$128\times 128$]{\includegraphics[height=\picw,width=\picw]{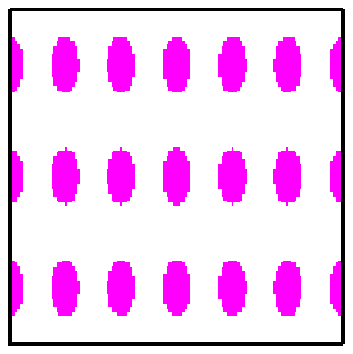}}
\caption{Cells marked by error estimator with $\psi = 20\%$ in Example~\ref{ex:sin_6pix_3piy}, Case~\ref{ex:case_0}, $\hat{V}_{h} = \caS^{2,2}_{h}, \hat{Y}_h = \caS^{3,2}_{h} \otimes \caS^{2,3}_{h}$.\label{fig_nexSU_0}}
\end{figure}

\begin{table}[!ht]\begin{center}
\begin{tabular}{|c|rr|rrr|rrr|rrr|c}\hline
mesh-size & 
\multicolumn{2}{|c|}{$\#$DOF} & 
\multicolumn{3}{|c|}{assembling-time} & 
\multicolumn{3}{|c|}{solving-time} & 
\multicolumn{3}{|c|}{sum} 
\\
 ~ & $u_h$ & $y_h$ 
& \multicolumn{1}{|c}{\emph{pde}}
& \multicolumn{1}{c}{\emph{est}}
& \multicolumn{1}{c|}{$\frac{\emph{est}}{\emph{pde}}$}
& \multicolumn{1}{|c}{\emph{pde}}
& \multicolumn{1}{c}{\emph{est}} 
& \multicolumn{1}{c|}{$\frac{\emph{est}}{\emph{pde}}$} 
& \multicolumn{1}{|c}{\emph{pde}} 
& \multicolumn{1}{c}{\emph{est}} 
& \multicolumn{1}{c|}{$\frac{\emph{est}}{\emph{pde}}$} 
\\\hline 
$8 \times 8$ & 100 & 220 & 0.04 & 0.17 & { 4.39 } & $<$0.01 & $<$0.01 & { 5.16 } & 0.04 & 0.17 & { 4.40 } \\
$16 \times 16$ & 324 & 684 & 0.14 & 0.59 & { 4.25 } & $<$0.01 & 0.01 & { 5.39 } & 0.14 & 0.60 & { 4.26 } \\
$32 \times 32$ & 1156 & 2380 & 0.46 & 2.17 & { 4.70 } & 0.01 & 0.03 & { 4.71 } & 0.47 & 2.20 & { 4.70 } \\
$64 \times 64$ & 4356 & 8844 & 1.82 & 8.51 & { 4.68 } & 0.03 & 0.20 & { 6.15 } & 1.85 & 8.70 & { 4.70 } \\
$128 \times 128$ & 16900 & 34060 & 7.38 & 34.19 & { 4.63 } & 0.15 & 0.87 & { 5.70 } & 7.54 & 35.06 & { 4.65 } \\
$256 \times 256$ & 66564 & 133644 & 33.30 & 149.78 & { 4.50 } & 0.84 & 5.66 & { 6.78 } & 34.14 & 155.44 & { 4.55 } \\
$512 \times 512$ & 264196 & 529420 & 191.11 & 766.10 & { 4.01 } & 3.77 & 33.92 & { 9.00 } & 194.88 & 800.03 & { 4.11 } \\
\hline\end{tabular}
\caption{Number of DOF and timings in Example~\ref{ex:sin_6pix_3piy}, Case~\ref{ex:case_0}, $\hat{V}_{h} = \caS^{2,2}_{h}, \hat{Y}_h = \caS^{3,2}_{h} \otimes \caS^{2,3}_{h}$. \label{nexSUtime_0}}
\end{center}\end{table}
\subsection{Alternative Cost-Efficient Procedure}
\label{sec_NumRealIgA}

Recall that the cost of Step~1 of the algorithm presented in Section~\ref{sec_minmaj} depends on the choice of $Y_h\subset H(\Omega,\dvg)$. As shown in Lemma~\ref{lem_Moconv}, we can make the estimate as sharp as we desire by choosing a suitably large space $Y_h$.
However, the larger $Y_h$ is chosen, the more costly setting up and solving the system \eqref{e_Lhyhrh} becomes. Clearly, it is highly desirable to keep the cost for error estimation below the cost for solving the original problem. 

As discussed above, choosing $\hat{Y}_h$ as in \eqref{e_Yh0} does not result in a cost-efficient method. Apart from the fact that $y_h$ is vector-valued while $u_h$ is scalar, another aspect contributes to the high cost for the procedure presented in Section~\ref{sec_proc0}.
Recall that, by choosing $\hat{Y}_h$ as in \eqref{e_Yh0}, we have
\begin{align*}
y_1 & \in \caS^{p+1,p}_{h},\\
y_2 & \in \caS^{p,p+1}_{h},
\end{align*}
i.e., the components of $y_h$ are in different spline spaces. Hence, we have to compute different basis functions for $y_1$ and $y_2$ (note that this can be a costly procedure for higher polynomial degrees). Furthermore, when assembling, for example, the matrix $\underline{L}_h^1$, we need to compute integrals over products of basis functions of the form
\begin{equation*}
\int_\Omega R_i R_j \dx.
\end{equation*}
With $\hat{Y}_h$ as  in \eqref{e_Yh0}, the product $R_i R_j$ of basis functions of $y_1$ is different to the product of basis functions of $y_2$, hence, the integrals have to be evaluated independently for $y_1$ and $y_2$.
\begin{case}\label{ex:case_1}
In the light of these observations, and since $(C^{p-2})^d \subset H(\Omega, \dvg),\ \forall p \geq 2$, we study the following alternative choice for~$\hat{Y}_h$.
\begin{equation}
\hat{Y}_h = \caS^{p+1,p+1}_{h} \otimes \caS^{p+1,p+1}_{h}.
\label{e_Yh1}
\end{equation}
Thereby, we choose a function space $\hat{Y}_h$ on the parameter domain and, analogously to the relation of $\hat{V}_h$ and $V_h$, we define the function space $Y_h$ by the push-forward
\[ Y_h = \hat{Y}_h \circ G^{-1}.\]
\end{case}
We refer to this setting as \emph{Case~\ref{ex:case_1}} in the remainder of the paper. With this choice, $y_1$ and $y_2$ are contained in the same spline spaces. Hence, the basis functions need to be computed only once, and any computed function values can be used for both components of $y_h$.
The computed efficiency indices are presented in Table~\ref{nexSUeff_1}, which show that we obtain even better (i.e., sharper) upper bounds for the exact error with $\hat{Y}_h$ as in \eqref{e_Yh1} than with the choice \eqref{e_Yh0}. When we compare the plots of the cells marked by the error estimator in Figure~\ref{fig_nexSU_1} to the plots in Figure~\ref{fig_nexSU_Ex}, we see that the error distribution is again captured accurately starting from the mesh $32\times 32$. The dashed line in Table~\ref{nexSUeff_1} indicates that the criterion \eqref{e_reli} is fulfilled with $C_\oplus = 5$ starting from the mesh $64\times 64$.
%

%
The timings obtained with this method are presented in Table~\ref{nexSUtime_1}. This approach reduced the total time needed for computing the majorant from a factor of about $4.5$ to a factor of approximately $3$ compared to the time for assembling and solving the original problem. Nevertheless, a factor of $3$ in the timings is still not very appealing, and demands further reduction in the cost.
\begin{table}[!ht]\begin{center}
\begin{tabular}{|c|r|rr|}\hline
 mesh-size & \multicolumn{1}{|c|}{$I_{\text{eff}}$}& \multicolumn{1}{|c}{$a_1 B_1$}
& \multicolumn{1}{c|}{$a_2 B_2$}
\\
\hline
\hline
 $8 \times 8$  & { 2.77 } & 8.08e+01 & 1.24e+01 \\
 $16 \times 16$  & { 1.71 } & 5.75e-01 & 3.96e-01 \\
 $32 \times 32$  & { 1.32 } & 2.14e-02 & 7.05e-03 \\ \hdashline 
 $64 \times 64$  & { 1.16 } & 1.11e-03 & 1.78e-04 \\
 $128 \times 128$  & { 1.08 } & 6.39e-05 & 5.08e-06 \\
 $256 \times 256$  & { 1.04 } & 3.83e-06 & 1.53e-07 \\
 $512 \times 512$  & { 1.02 } & 2.35e-07 & 4.69e-09 \\
\hline
\end{tabular}
\caption{Efficiency index and components of the majorant in Example~\ref{ex:sin_6pix_3piy}, Case~\ref{ex:case_1}, $\hat{V}_{h} = \caS^{2,2}_{h}, \hat{Y}_h = \caS^{3,3}_{h} \otimes \caS^{3,3}_{h}$.\label{nexSUeff_1}}
\end{center}\end{table}

\begin{figure}[!ht]\centering
\subfigure[$16\times 16$]{\includegraphics[height=\picw,width=\picw]{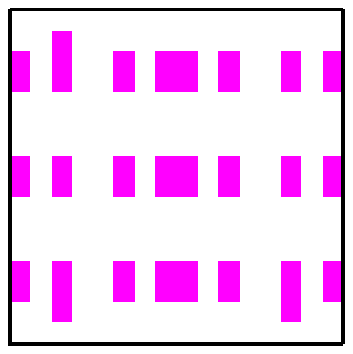}} \qquad
\subfigure[$32\times 32$\label{nexSU_32_0_1}]{\includegraphics[height=\picw,width=\picw]{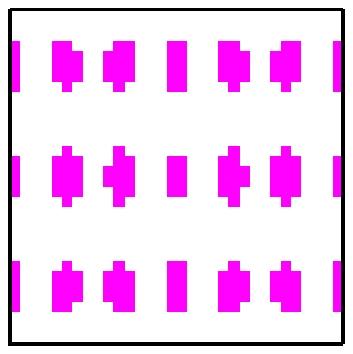}} \qquad
\subfigure[$64\times 64$]{\includegraphics[height=\picw,width=\picw]{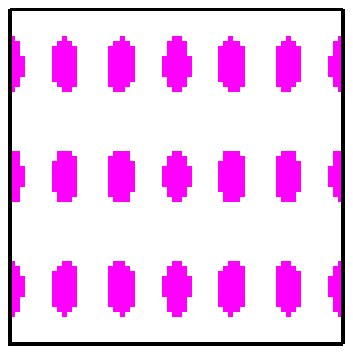}} \qquad
\subfigure[$128\times 128$]{\includegraphics[height=\picw,width=\picw]{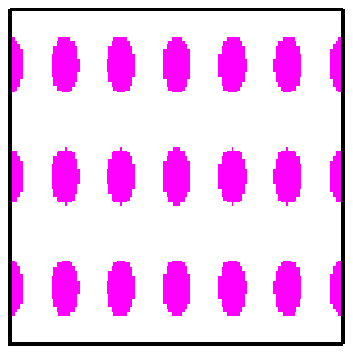}}
\caption{Cells marked by error estimator with $\psi = 20\%$ in Example~\ref{ex:sin_6pix_3piy}, Case~\ref{ex:case_1}, $\hat{V}_{h} = \caS^{2,2}_{h}, \hat{Y}_h = \caS^{3,3}_{h} \otimes \caS^{3,3}_{h}$.\label{fig_nexSU_1}}
\end{figure}

\begin{table}[!ht]\begin{center}
\begin{tabular}{|c|rr|rrr|rrr|rrr|}\hline
mesh-size & 
\multicolumn{2}{|c|}{$\#$DOF} & 
\multicolumn{3}{|c|}{assembling-time} & 
\multicolumn{3}{|c|}{solving-time} & 
\multicolumn{3}{|c|}{sum} 
\\
& $u_h$ & $y_h$ 
& \multicolumn{1}{|c}{\emph{pde}} 
& \multicolumn{1}{c}{\emph{est}} 
& \multicolumn{1}{c|}{$\frac{\emph{est}}{\emph{pde}}$} 
& \multicolumn{1}{|c}{\emph{pde}} 
& \multicolumn{1}{c}{\emph{est}} 
& \multicolumn{1}{c|}{$\frac{\emph{est}}{\emph{pde}}$} 
& \multicolumn{1}{|c}{\emph{pde}} 
& \multicolumn{1}{c}{\emph{est}} 
& \multicolumn{1}{c|}{$\frac{\emph{est}}{\emph{pde}}$} 
\\
\hline
\hline
 $8 \times 8$ & 100 & 242 & 0.04 & 0.11 & { 2.78 } & $<$0.01 & $<$0.01 & { 1.51 } & 0.04 & 0.11 & { 2.76 } \\
 $16 \times 16$ & 324 & 722 & 0.12 & 0.34 & { 2.86 } & $<$0.01 & 0.01 & { 5.33 } & 0.12 & 0.35 & { 2.90 } \\
 $32 \times 32$ & 1156 & 2450 & 0.46 & 1.35 & { 2.94 } & 0.01 & 0.05 & { 7.69 } & 0.47 & 1.40 & { 3.01 } \\
 $64 \times 64$ & 4356 & 8978 & 1.77 & 5.30 & { 2.99 } & 0.03 & 0.27 & { 8.02 } & 1.80 & 5.57 & { 3.09 } \\
 $128 \times 128$ & 16900 & 34322 & 7.39 & 21.89 & { 2.96 } & 0.16 & 1.45 & { 9.26 } & 7.55 & 23.34 & { 3.09 } \\
 $256 \times 256$ & 66564 & 134162 & 33.00 & 94.69 & { 2.87 } & 0.84 & 8.83 & { 10.54 } & 33.84 & 103.52 & { 3.06 } \\
 $512 \times 512$ & 264196 & 530450 & 191.59 & 498.20 & { 2.60 } & 3.83 & 61.45 & { 16.06 } & 195.42 & 559.65 & { 2.86 } \\
\hline\end{tabular}
\caption{Number of DOF and timings in Example~\ref{ex:sin_6pix_3piy}, Case~\ref{ex:case_1}, $\hat{V}_{h} = \caS^{2,2}_{h}, \hat{Y}_h = \caS^{3,3}_{h} \otimes \caS^{3,3}_{h}$.\label{nexSUtime_1}}
\end{center}\end{table}
\begin{remark}\label{rem_C0FEM}
Note that the use of equal degree polynomials for both the components of $\hat{Y}_{h}$ is only possible because of extra continuity readily available from NURBS basis functions. A counter-part is not possible in FEM case simply because the derivatives of FEM basis functions (with $C^{0}$-continuity) is only in $L^{2}$, and hence, one can not avoid using proper subspaces of $H(\Omega,\dvg)$, e.g., Raviart-Thomas space (with unequal degree polynomials in both the dimensions for both the components). It is further important to note from a close inspection of Tables~\ref{nexSUeff_0} and \ref{nexSUeff_1} that the results from equal degree components of vector-valued quantity outperformed the results from unequal degree case.
\end{remark}

In order to further reduce the computational cost, we reduce the number of DOFs of $y_h$ by coarsening the mesh by a factor $K$ in each dimension. The number of DOFs of $y_h$ is thus reduced to $2 N_u/K^2$ (asymptotically). The larger $K$ is chosen, the greater the reduction of DOFs will be. At the same time, if the coarsening is done too aggressively, sharp features might not be detected properly on coarse meshes. We counter the reduction in accuracy due to mesh-coarsening by increasing the polynomial degree of $y_h$ by some positive integer $k$, i.e., we choose
\begin{equation}
\hat{Y}_h = \caS^{p+k,p+k}_{Kh} \otimes \caS^{p+k,p+k}_{Kh}.
\label{e_Yhk}
\end{equation}
Note that, if desired, one could also choose different factors $K_1$ and $K_2$ and different degree increases $k_1$ and $k_2$ for the first and second component, respectively.
\begin{remark}\label{rem_choiceYh}
With the choices of $\hat{Y}_h$ as in \eqref{e_Yhk}, we take advantage of the following specific property of univariate NURBS basis functions. For $C^{p-1}$-continuity, increasing the polynomial degree by $k$ only adds a total of $k$ additional basis functions. In other words, the global smoothness can be increased at the cost of only a few additional DOFs. Coarsening the mesh by a factor $K$, however, will also reduce the number of DOFs by the same factor $K$ (asymptotically).

Moreover, as we will see from the three cases of Example~\ref{ex:sin_6pix_3piy}, asymptotically we get better efficiency indices with higher degree $p$ and coarser meshes as compared to lower degree $p$ and finer meshes. This phenomenon is similar to the $p$ finite element discretization for problems with smooth solutions, where increasing the polynomial degree for a fixed mesh size $h$ is much more advantageous than decreasing the mesh size $h$ for a fixed (low) polynomial degree. Nevertheless, such a low cost construction for higher degree $p$ is not possible in FEM discretizations.
\end{remark}
Note that Case~\ref{ex:case_1} discussed above fits into this framework, since Case~\ref{ex:case_1} corresponds to the choice $K = k = 1$.
\begin{case}\label{ex:case_2}
For the next setting, we apply moderate mesh-coarsening by choosing
\begin{equation*}
K = k = 2\text{\ (i.e.,\ }\hat{Y}_h = \caS^{p+2,p+2}_{2h} \otimes \caS^{p+2,p+2}_{2h}).
\end{equation*}
\end{case}
Similar to \emph{Case~\ref{ex:case_1}}, the function space $Y_h$ is defined by the push-forward
\[ Y_h = \hat{Y}_h \circ G^{-1}.\]
This setting will be referred to as \emph{Case~\ref{ex:case_2}} in the remainder of the paper. The computed efficiency indices along with the magnitudes of the terms $a_1 B_1$ and $a_2 B_2$ for Case~\ref{ex:case_2} are presented in Table~\ref{nexSUeff_2}, and the marked cells are plotted in Figure~\ref{fig_nexSU_2}. The dashed line indicates that criterion \eqref{e_reli} is fulfilled with $C_\oplus = 5$, and that a good upper bound of the error is computed and the correct error distribution is captured on meshes starting from $64\times 64$. On coarse meshes, however, the efficiency index is larger than in Case~\ref{ex:case_1}, which is due to the boundary effects.
%
%
The timings presented in Table~\ref{nexSUtime_2} show that, even though Case~\ref{ex:case_2} is faster than Case~\ref{ex:case_1}, this approach still costs roughly as much as solving the original problem. This is due to the costlier evaluation of the higher degree basis functions, as well as the increased support and overlap of the basis functions, which results in more non-zero entries in $\underline{L}_h$ than in $\underline{K}_h$.
\begin{table}[!ht]\begin{center}
\begin{tabular}{|c|r|rr|}\hline
 mesh-size & \multicolumn{1}{|c|}{$I_{\text{eff}}$}& \multicolumn{1}{|c}{$a_1 B_1$}
& \multicolumn{1}{c|}{$a_2 B_2$}
\\
\hline
\hline
 $8 \times 8$  & { 14.19 } & 1.59e+03 & 8.53e+02 \\ 
 $16 \times 16$  & { 8.49 } & 1.97e+01 & 4.32e+00 \\
 $32 \times 32$  & { 1.82 } & 3.05e-02 & 2.41e-02 \\ \hdashline
 $64 \times 64$  & { 1.16 } & 1.12e-03 & 1.76e-04 \\ 
 $128 \times 128$  & { 1.04 } & 6.14e-05 & 2.24e-06 \\ 
 $256 \times 256$  & { 1.01 } & 3.72e-06 & 3.32e-08 \\ 
 $512 \times 512$  & { 1.00 } & 2.31e-07 & 5.13e-10 \\ 
\hline
\end{tabular}
\caption{Efficiency index and components of the majorant in Example~\ref{ex:sin_6pix_3piy}, Case~\ref{ex:case_2}, $\hat{V}_{h} = \caS^{2,2}_{h}, \hat{Y}_h = \caS^{4,4}_{2h} \otimes \caS^{4,4}_{2h}$.\label{nexSUeff_2}}
\end{center}\end{table}

\setlength{\picw}{3cm}
\begin{figure}[!ht]\centering
\subfigure[$16\times 16$]{\includegraphics[height=\picw,width=\picw]{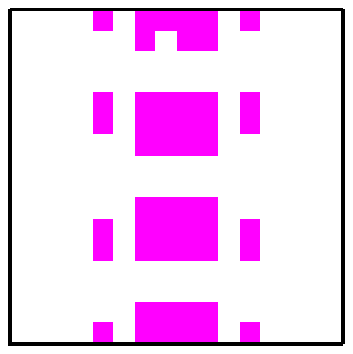}} \qquad
\subfigure[$32\times 32$\label{nexSU_32_1}]{\includegraphics[height=\picw,width=\picw]{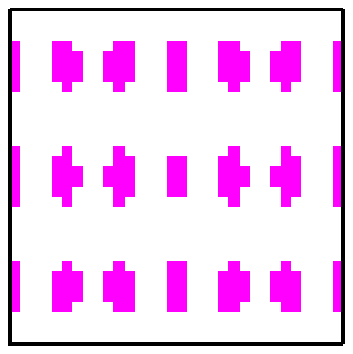}} \qquad
\subfigure[$64\times 64$]{\includegraphics[height=\picw,width=\picw]{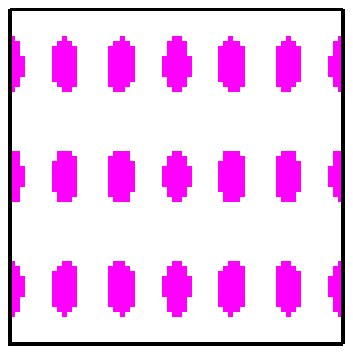}} \qquad
\subfigure[$128\times 128$]{\includegraphics[height=\picw,width=\picw]{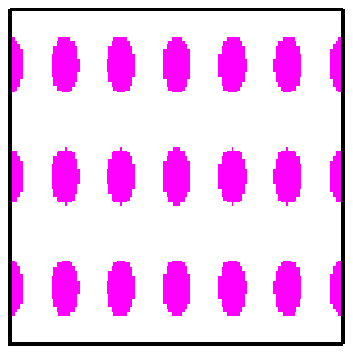}} 
\caption{Cells marked by error estimator with $\psi = 20\%$ in Example~\ref{ex:sin_6pix_3piy}, Case~\ref{ex:case_2}, $\hat{V}_{h} = \caS^{2,2}_{h}, \hat{Y}_h = \caS^{4,4}_{2h} \otimes \caS^{4,4}_{2h}$.\label{fig_nexSU_2}}
\end{figure}

\begin{table}[!ht]\begin{center}
\begin{tabular}{|c|rr|rrr|rrr|rrr|}\hline
 mesh-size & 
\multicolumn{2}{|c|}{$\#$DOF} & 
\multicolumn{3}{|c|}{assembling-time} & 
\multicolumn{3}{|c|}{solving-time} & 
\multicolumn{3}{|c|}{sum} 
\\
 & $u_h$ & $y_h$ 
& \multicolumn{1}{|c}{\emph{pde}} 
& \multicolumn{1}{c}{\emph{est}} 
& \multicolumn{1}{c|}{$\frac{\emph{est}}{\emph{pde}}$} 
& \multicolumn{1}{|c}{\emph{pde}} 
& \multicolumn{1}{c}{\emph{est}} 
& \multicolumn{1}{c|}{$\frac{\emph{est}}{\emph{pde}}$} 
& \multicolumn{1}{|c}{\emph{pde}} 
& \multicolumn{1}{c}{\emph{est}} 
& \multicolumn{1}{c|}{$\frac{\emph{est}}{\emph{pde}}$} 
\\
\hline
\hline
 $8 \times 8$ & 100 & 128 & 0.03 & 0.05 & { 1.39 } & $<$0.01 & $<$0.01 & { 1.16 } & 0.04 & 0.05 & { 1.39 } \\ 
 $16 \times 16$ & 324 & 288 & 0.14 & 0.18 & { 1.29 } & $<$0.01 & $<$0.01 & { 0.92 } & 0.14 & 0.18 & { 1.28 } \\ 
 $32 \times 32$ & 1156 & 800 & 0.54 & 0.59 & { 1.10 } & 0.01 & 0.02 & { 2.32 } & 0.55 & 0.61 & { 1.11 } \\ 
 $64 \times 64$ & 4356 & 2592 & 1.91 & 2.33 & { 1.22 } & 0.04 & 0.08 & { 2.09 } & 1.95 & 2.40 & { 1.23 } \\ 
 $128 \times 128$ & 16900 & 9248 & 7.46 & 9.54 & { 1.28 } & 0.19 & 0.51 & { 2.75 } & 7.64 & 10.05 & { 1.32 } \\ 
 $256 \times 256$ & 66564 & 34848 & 33.93 & 39.02 & { 1.15 } & 0.90 & 2.59 & { 2.88 } & 34.82 & 41.60 & { 1.19 } \\ 
 $512 \times 512$ & 264196 & 135200 & 196.23 & 177.98 & { 0.91 } & 4.08 & 15.91 & { 3.90 } & 200.31 & 193.89 & { 0.97 } \\ 
\hline\end{tabular}
\caption{Number of DOF and timings in Example~\ref{ex:sin_6pix_3piy}, Case~\ref{ex:case_2}, $\hat{V}_{h} = \caS^{2,2}_{h}, \hat{Y}_h = \caS^{4,4}_{2h} \otimes \caS^{4,4}_{2h}$.\label{nexSUtime_2}}
\end{center}\end{table}
\begin{case}\label{ex:case_3}
To further improve the timings, we coarsen the mesh more aggressively by a factor of $4$ and, at the same time, increase the polynomial degree of $y_h$ by 4, as compared to $u_h$, i.e., 
\begin{equation*}
K = k = 4 \text{\ (i.e.,\ }\hat{Y}_h = \caS^{p+4,p+4}_{4h} \otimes \caS^{p+4,p+4}_{4h}).
\end{equation*}
\end{case}
Again, similar to \emph{Case~\ref{ex:case_1}}, the function space $Y_h$ is defined by the push-forward
\[ Y_h = \hat{Y}_h \circ G^{-1}.\]
We refer to this setting as \emph{Case~\ref{ex:case_3}} in the remainder of the paper. This aggressive coarsening notably affects the efficiency index on coarse meshes, see Table~\ref{nexSUeff_3}. On fine meshes, however, the efficiency indices are close to 1 in all presented cases. The number of DOFs of $y_h$ in Case~\ref{ex:case_3} is only $N_u/8$ (asymptotically).
%
The timings presented in Table~\ref{nexSUtime_3} show that this setting results in a method which can be performed significantly faster (at almost half of the cost) than solving the original problem.
The more aggressive reduction of DOF outweighs the additional costs mentioned above, even though the polynomial degree is now increased by 4.
\begin{remark}\label{rmk_Coplus}
In all Cases for Example~\ref{ex:sin_6pix_3piy}, criterion \eqref{e_reli} is fulfilled with $C_\oplus = 5$ on meshes of size $64 \times 64$ and finer. This is indicated by the dashed lines in Tables~\ref{nexSUeff_0}, \ref{nexSUeff_1}, \ref{nexSUeff_2} and \ref{nexSUeff_3}, and is clear from Figures~\ref{fig_nexSU_0}-\ref{fig_nexSU_3}. Therefore, Example~\ref{ex:sin_6pix_3piy} and the examples discussed in Section~\ref{sec_numex} show that $C_\oplus = 5$ is a good choice for checking criterion \eqref{e_reli} numerically, even though this choice may be conservative in some cases.
\end{remark}
\begin{table}[!ht]\begin{center}
\begin{tabular}{|c|r|rr|}\hline
 mesh-size & \multicolumn{1}{|c|}{$I_{\text{eff}}$}& \multicolumn{1}{|c}{$a_1 B_1$}
& \multicolumn{1}{c|}{$a_2 B_2$}
\\
\hline
\hline
 $8 \times 8$  & { 11.28 } & 5.38e+02 & 1.01e+03 \\ 
 $16 \times 16$  & { 36.43 } & 2.83e+02 & 1.60e+02 \\ 
 $32 \times 32$  & { 12.63 } & 2.04e+00 & 5.81e-01 \\\hdashline 
 $64 \times 64$  & { 1.17 } & 1.13e-03 & 1.88e-04 \\ 
 $128 \times 128$  & { 1.01 } & 5.98e-05 & 3.79e-07 \\ 
 $256 \times 256$  & { 1.00 } & 3.70e-06 & 1.24e-09 \\ 
 $512 \times 512$  & { 1.00 } & 2.31e-07 & 5.32e-12 \\ 
\hline
\end{tabular}
\caption{Efficiency index and components of the majorant in Example~\ref{ex:sin_6pix_3piy}, Case~\ref{ex:case_3}, $\hat{V}_{h} = \caS^{2,2}_{h}, \hat{Y}_h = \caS^{6,6}_{4h} \otimes \caS^{6,6}_{4h}$.\label{nexSUeff_3}}
\end{center}\end{table}

\begin{figure}[!ht]\centering
\subfigure[$16\times 16$]{\includegraphics[height=\picw,width=\picw]{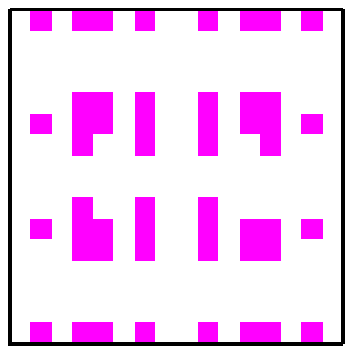}}\qquad
\subfigure[$32\times 32$\label{nexSU_32_2}]{\includegraphics[height=\picw,width=\picw]{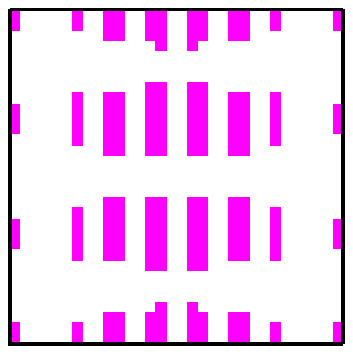}}\qquad
\subfigure[$64\times 64$]{\includegraphics[height=\picw,width=\picw]{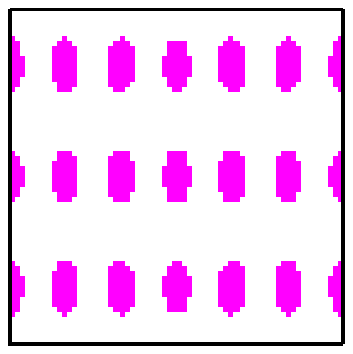}}\qquad
\subfigure[$128\times 128$]{\includegraphics[height=\picw,width=\picw]{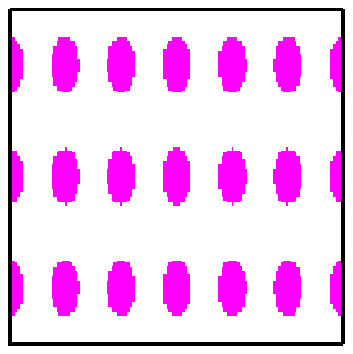}}
\caption{Cells marked by error estimator with $\psi = 20\%$ in Example~\ref{ex:sin_6pix_3piy}, Case~\ref{ex:case_3}, $\hat{V}_{h} = \caS^{2,2}_{h}, \hat{Y}_h = \caS^{6,6}_{4h} \otimes \caS^{6,6}_{4h}$.\label{fig_nexSU_3}}
\end{figure}

\begin{table}[!ht]\begin{center}
\begin{tabular}{|c|rr|rrr|rrr|rrr|}\hline
 mesh-size & 
\multicolumn{2}{|c|}{$\#$DOF} & 
\multicolumn{3}{|c|}{assembling-time} & 
\multicolumn{3}{|c|}{solving-time} & 
\multicolumn{3}{|c|}{sum} 
\\
& $u_h$ & $y_h$ 
& \multicolumn{1}{|c}{\emph{pde}} 
& \multicolumn{1}{c}{\emph{est}} 
& \multicolumn{1}{c|}{$\frac{\emph{est}}{\emph{pde}}$} 
& \multicolumn{1}{|c}{\emph{pde}} 
& \multicolumn{1}{c}{\emph{est}} 
& \multicolumn{1}{c|}{$\frac{\emph{est}}{\emph{pde}}$} 
& \multicolumn{1}{|c}{\emph{pde}} 
& \multicolumn{1}{c}{\emph{est}} 
& \multicolumn{1}{c|}{$\frac{\emph{est}}{\emph{pde}}$} 
\\
\hline
\hline
 $8 \times 8$ & 100 & 128 & 0.04 & 0.03 & { 0.76 } & $<$0.01 & $<$0.01 & { 1.09 } & 0.04 & 0.03 & { 0.76 } \\ 
 $16 \times 16$ & 324 & 200 & 0.14 & 0.10 & { 0.69 } & $<$0.01 & $<$0.01 & { 0.61 } & 0.14 & 0.10 & { 0.69 } \\ 
 $32 \times 32$ & 1156 & 392 & 0.54 & 0.31 & { 0.57 } & 0.01 & $<$0.01 & { 0.34 } & 0.55 & 0.31 & { 0.57 } \\ 
 $64 \times 64$ & 4356 & 968 & 1.90 & 1.19 & { 0.63 } & 0.04 & 0.01 & { 0.26 } & 1.94 & 1.20 & { 0.62 } \\ 
 $128 \times 128$ & 16900 & 2888 & 7.49 & 4.86 & { 0.65 } & 0.16 & 0.14 & { 0.84 } & 7.66 & 4.99 & { 0.65 } \\ 
 $256 \times 256$ & 66564 & 9800 & 33.90 & 20.15 & { 0.59 } & 0.91 & 0.82 & { 0.91 } & 34.81 & 20.98 & { 0.60 } \\ 
 $512 \times 512$ & 264196 & 35912 & 194.25 & 84.70 & { 0.44 } & 4.10 & 5.45 & { 1.33 } & 198.35 & 90.15 & { 0.45 } \\ 
\hline\end{tabular}
\caption{Number of DOF and timings in Example~\ref{ex:sin_6pix_3piy}, Case~\ref{ex:case_3}, $\hat{V}_{h} = \caS^{2,2}_{h}, \hat{Y}_h = \caS^{6,6}_{4h} \otimes \caS^{6,6}_{4h}$.\label{nexSUtime_3}}
\end{center}\end{table}
We now comment on the interleaved iterations. The results in the Tables~\ref{nexSUeff_0}-\ref{nexSUeff_3} were obtained by applying only two interleaved iterations, as described in Section~\ref{sec_minmaj}. As mentioned there, a sufficiently accurate result can be obtained already after the first such iteration. To illustrate this, we present the efficiency indices for Case~\ref{ex:case_3} in Table~\ref{nexSUbetaloops}, which were obtained after one, two, and four interleaved iterations, respectively. The efficiency index does vary notably on the coarser meshes, but since all of these values greatly overestimate the exact error, they do not correctly capture the error distribution. On meshes, where the criterion \eqref{e_reli} is fulfilled with $C_\oplus = 5$, and thus the error distribution is correctly recovered, the differences due to more interleaved iterations are insignificant.

\begin{table}[!ht]\begin{center}
\begin{tabular}{|c|rrr|}\hline
mesh-size & \multicolumn{3}{|c|}{interleaved iterations}\\
& 
\multicolumn{1}{|c}{1} & 
\multicolumn{1}{c}{2} &
\multicolumn{1}{c|}{4}  \\
\hline
\hline
 $8 \times 8$     & 11.84	& 11.28	& 11.25 \\
 $16 \times 16$   & 80.31	& 36.43	& 33.78 \\
 $32 \times 32$   & 17.36	& 12.63	& 10.11 \\\hdashline
 $64 \times 64$   & 1.20	& 1.17	& 1.17  \\
 $128 \times 128$ & 1.01	& 1.01 	& 1.01 \\
 $256 \times 256$ & 1.00	& 1.00	& 1.00 \\
 $512 \times 512$ & 1.00	& 1.00	& 1.00 \\
\hline
\end{tabular}\caption{Comparison of $I_\text{eff}$ for different numbers of interleaved iterations, Example~\ref{ex:sin_6pix_3piy}, Case~\ref{ex:case_3}, $\hat{V}_{h} = \caS^{2,2}_{h}, \hat{Y}_h = \caS^{6,6}_{4h} \otimes \caS^{6,6}_{4h}$.\label{nexSUbetaloops}}
\end{center}\end{table} 

\begin{remark}\label{rmk_balance}
The observations discussed above illustrate that one has to balance the sharpness of the majorant on the one hand, and the required computational effort on the one hand. Note that in typical practical applications, the exact solution (and thus the sharpness of the majorant) is not known. Therefore, to address the balance between sharpness and required computational effort, we propose the following strategy. If the mesh is coarse and the total computational cost for the error estimate is moderate, we apply no (or only moderate) coarsening. When the original mesh is fine (problem size being large), we coarsen the mesh aggressively, and thereby, profit from the fast computation of the estimate. While exercising this strategy it is important to enforce the criterion \eqref{e_reli} with $C_\oplus \ge 5$.
\end{remark}
\section{Numerical Examples}
\label{sec_numex}

In this section, we present further numerical examples which illustrate the potential of the proposed a posteriori error estimator. We will present the results corresponding to the three settings discussed in Section~\ref{sec_effcomp}, namely \emph{Case~\ref{ex:case_1}}, \emph{Case~\ref{ex:case_2}} and \emph{Case~\ref{ex:case_3}} with the choices $K = k = 1$, $K = k = 2$, and $K = k = 4$, respectively. As in Example~\ref{ex:sin_6pix_3piy}, the initial guess for $\beta$ is $0.01$.
As discussed in Section~\ref{sec_Bsplines}, the parameter domain in all presented examples is the unit square $\hat{\Omega} = (0,1)^2$. The mesh-sizes in the two coordinate directions, which will be presented in the tables, are determined by the respective initial meshes, which in turn, are determined by the geometry mappings.
The figures plotted in black represent the computations based on the exact error, and the figures plotted in magenta represent the computations based on the majorant. {The data presented in the tables is as described in the beginning of Section~\ref{sec_effcomp}.}

We first consider an example with reduced continuity $C^{p-m},\ m>1$.

\begin{example}\label{ex:sin_6pix_3piy_p4}
\textbf{Sinus function in a unit square with $p=q=4$ and $C^{1}$-continuity:}
We consider the same exact solution and the same physical domain as in Example~\ref{ex:sin_6pix_3piy}, i.e.,
\begin{align*}
u = \sin( 6 \pi x) \sin(3 \pi y), \quad \Omega = (0,1)^2.
\end{align*}
However, we now use B-splines of degree $p=q=4$ to represent $\Omega$, and we add a triple knot at the coordinates $x = 0.5$ and $y = 0.5$. The initial knot vectors are thus given by
\[
s = t = (0,0,0,0,0,0.5,0.5,0.5,1,1,1,1,1),
\]
and the geometry mapping is only $C^1$-continuous at the coordinate $0.5$.
\end{example}
\begin{table}[!ht]\begin{center}
\begin{tabular}{|c|r|rr|}\hline
 mesh-size & \multicolumn{1}{|c|}{$I_{\text{eff}}$}& \multicolumn{1}{|c}{$a_1 B_1$}
& \multicolumn{1}{c|}{$a_2 B_2$}\\
\hline
\multicolumn{4}{c}{Case~\ref{ex:case_1}}\\
\hline
$18 \times 18$  & { 1.84 } & 1.04e-03 & 9.00e-04   \\
$34 \times 34$  & { 1.40 } & 1.78e-06 & 7.23e-07   \\\hdashline
$66 \times 66$  & { 1.20 } & 5.09e-09 & 1.00e-09   \\
$130 \times 130$  & { 1.10 } & 1.77e-11 & 1.74e-12 \\
$258 \times 258$  & { 1.05 } & 6.61e-14 & 3.25e-15 \\
\hline
\multicolumn{4}{c}{Case~\ref{ex:case_2}}\\
\hline
$18 \times 18$  & { 15.43} & 7.95e-02 & 5.75e-02  \\
$34 \times 34$  & { 6.04 } & 1.14e-05 & 3.53e-05   \\
$66 \times 66$  & { 1.76 } & 7.52e-09 & 5.69e-09   \\\hdashline
$130 \times 130$  & { 1.16 } & 1.87e-11 & 3.01e-12 \\
$258 \times 258$  & { 1.04 } & 6.54e-14 & 2.49e-15 \\
\hline
\multicolumn{4}{c}{Case~\ref{ex:case_3}}\\
\hline
$18 \times 18$  & { 132.77 } & 7.38e+00 & 2.76e+00 \\
$34 \times 34$  & { 148.41 } & 1.86e-02 & 9.53e-03 \\
$66 \times 66$  & { 6.42 } & 5.49e-08 & 1.21e-07   \\\hdashline
$130 \times 130$  & { 1.13 } & 1.83e-11 & 2.39e-12 \\
$258 \times 258$  & { 1.01 } & 6.34e-14 & 3.78e-16 \\
\hline
\end{tabular}
\caption{Efficiency index and components of the majorant in Example~\ref{ex:sin_6pix_3piy_p4},  $\hat{V}_{h} = \caS^{4,4}_{h}$ with $C^{1}$-continuity.\label{nexSUp4_eff}}
\end{center}\end{table}

\setlength{\picw}{3cm}
\begin{figure}[!ht]
\centering
\subfigure[$18\times 18$.]{\includegraphics[height=\picw,width=\picw]{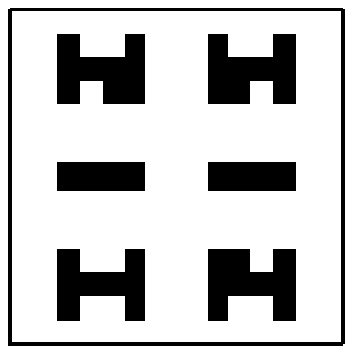}} \qquad
\subfigure[$34\times 34$.]{\includegraphics[height=\picw,width=\picw]{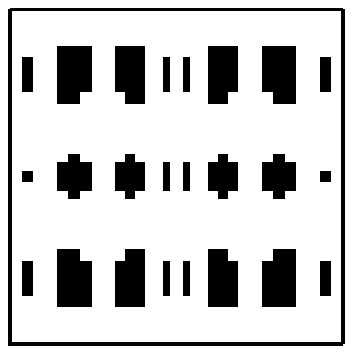}} \qquad
\subfigure[$66\times 66$.]{\includegraphics[height=\picw,width=\picw]{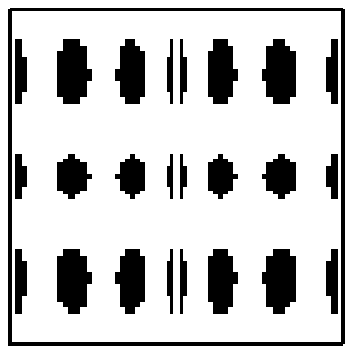}} \qquad
\subfigure[$130\times 130$.]{\includegraphics[height=\picw,width=\picw]{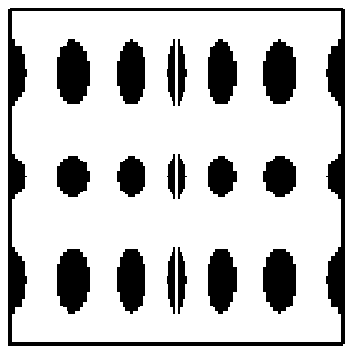}} \\
\caption{Cells marked by exact error with $\psi = 20\%$ in Example~\ref{ex:sin_6pix_3piy_p4},  $\hat{V}_{h} = \caS^{4,4}_{h}$ with $C^{1}$-continuity.\label{fig_nexSUp4_ex}}
\subfigure[$18\times 18$, Case~\ref{ex:case_1}.\label{fig_nexSUp4_est_a}]{\includegraphics[height=\picw,width=\picw]{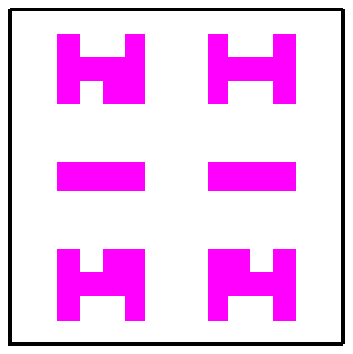}} \qquad
\subfigure[$34\times 34$, Case~\ref{ex:case_1}.\label{fig_nexSUp4_est_b}]{\includegraphics[height=\picw,width=\picw]{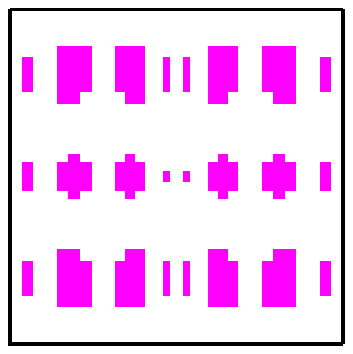}} \qquad
\subfigure[$66\times 66$, Case~\ref{ex:case_2}.\label{fig_nexSUp4_est_c}]{\includegraphics[height=\picw,width=\picw]{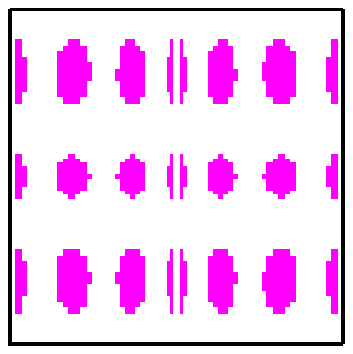}} \qquad
\subfigure[$130\times 130$, Case~\ref{ex:case_3}.\label{fig_nexSUp4_est_d}]{\includegraphics[height=\picw,width=\picw]{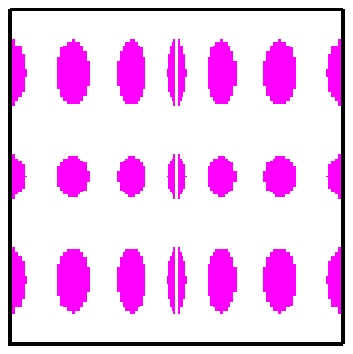}} \\
\caption{Cells marked by error estimator with $\psi = 20\%$ in Example~\ref{ex:sin_6pix_3piy_p4}, $\hat{V}_{h} = \caS^{4,4}_{h}$ with $C^{1}$-continuity.\label{fig_nexSUp4_est}}
\end{figure}
The computed efficiency indices are presented in Table~\ref{nexSUp4_eff}. The dashed lines, which correspond to criterion \eqref{e_reli} being fulfilled with $C_\oplus = 5$, again show that more aggressive mesh-coarsening requires a finer initial mesh. By this criterion, we get a good quality of the estimate and the indicated error distribution starting from the mesh $66\times 66$ in Case~\ref{ex:case_1}, and from $130 \times 130$ in Cases~\ref{ex:case_2} and~\ref{ex:case_3}.
We present the cells marked for refinement by the exact error in Figure~\ref{fig_nexSUp4_ex}, and the cells marked by the error estimator in Figure~\ref{fig_nexSUp4_est}. Figure~\ref{fig_nexSUp4_est_a} shows that the error distribution is already captured on the mesh $18\times 18$ in Case~\ref{ex:case_1}. In Case~\ref{ex:case_2}, we obtain a good indication of the error distribution from the mesh $66\times 66$, i.e., before criterion \eqref{e_reli} with $C_\oplus = 5$ is fulfilled.
Once the error distribution is captured correctly on a certain mesh, it is also captured on all finer meshes (as in Example~\ref{ex:sin_6pix_3piy}). Hence, we do not show all plots for all meshes and cases, but only the first meshes, on which the error distribution is captured correctly.
%
Also, we omit the presentation of the timings, since the overall behavior is as in Example~\ref{ex:sin_6pix_3piy}.
In the next example, we consider the case of non-trivial PDE coefficient matrix $A$.
\begin{example}\label{ex:sin_6pix_3piy_NTA}
Let the matrix $A$ be of the form of
\[
\left(
\begin{array}{cc}
e^{b_{11} x + b_{12} y} & 0 \\
0 & e^{b_{21} x + b_{22} y}
\end{array}
\right),
\]
which is positive definite for $b_{ij} \in \mathbb{R}^{+}$, $i, j = \{1, 2\}$. This will result in the PDE operator to be of the form of
\[
e^{b_{11} x + b_{12} y} \dfrac{\partial^{2}}{\partial x^{2}}
+ e^{b_{21} x + b_{22} y} \dfrac{\partial^{2}}{\partial y^{2}}
+ b_{11} e^{b_{11} x + b_{12} y} \dfrac{\partial}{\partial x}
+ b_{22} e^{b_{21} x + b_{22} y} \dfrac{\partial}{\partial y}.
\]
To have this PDE operator with full generality, we take $b_{11} = 0.1, b_{12} = 0.8, b_{21} = 0.4, b_{22} = 0.7$. With this generality, to have a good comparison of the efficiency indices with the examples considered so far, we again choose the exact solution to be $u = \sin (6 \pi x) \sin (3 \pi y)$. The right hand side function is accordingly calculated and the solution has homogeneous Dirichlet boundary values.
\end{example}
Note that in this case the constant $C_{\Omega}$ has to be accordingly modified. For the unit square domain, and the matrix $A$ given above, its value is taken as $\dfrac{c_{2}}{\pi \sqrt{2}}$, where
\[
c_{2} = \max \{e^{b_{11} x + b_{12} y}, e^{b_{21} x + b_{22} y}\}.
\]
\begin{table}[!ht]\begin{center}
\begin{tabular}{|c|r|rr|}\hline
mesh-size & \multicolumn{1}{|c|}{$I_{\text{eff}}$}& \multicolumn{1}{|c}{$a_1 B_1$}
& \multicolumn{1}{c|}{$a_2 B_2$}
\\
\hline
\multicolumn{4}{c}{Case~\ref{ex:case_1}}\\
\hline
 $8 \times 8$  & { 3.64 } & 2.02e+02 & 5.73e+01 \\ 
 $16 \times 16$  & { 6.00 } & 1.16e+01 & 7.75e+00 \\ 
 $32 \times 32$  & { 2.50 } & 6.82e-02 & 9.65e-02 \\
 $64 \times 64$  & { 1.74 } & 2.70e-03 & 1.98e-03 \\ 
 $128 \times 128$  & { 1.37 } & 1.31e-04 & 4.83e-05 \\ \hdashline 
 $256 \times 256$  & { 1.19 } & 7.04e-06 & 1.31e-06 \\ 
 $512 \times 512$  & { 1.09 } & 4.05e-07 & 3.77e-08 \\ 
\hline
\multicolumn{4}{c}{Case~\ref{ex:case_2}}\\
\hline
 $8 \times 8$  & { 38.29 } & 2.34e+04 & 5.36e+03 \\ 
 $16 \times 16$  & { 14.09 } & 9.07e+01 & 1.61e+01 \\ 
 $32 \times 32$  & { 4.64 } & 2.02e-01 & 3.67e-01 \\
 $64 \times 64$  & { 1.62 } & 2.53e-03 & 1.55e-03 \\ \hdashline 
 $128 \times 128$  & { 1.15 } & 1.10e-04 & 1.57e-05 \\ 
 $256 \times 256$  & { 1.04 } & 6.20e-06 & 2.16e-07 \\ 
 $512 \times 512$  & { 1.01 } & 3.77e-07 & 3.27e-09 \\ 
\hline
\multicolumn{4}{c}{Case~\ref{ex:case_3}}\\
\hline
 $8 \times 8$  & { 32.36 } & 2.81e+03 & 1.77e+04 \\ 
 $16 \times 16$  & { 122.17 } & 6.73e+03 & 1.30e+03 \\ 
 $32 \times 32$  & { 23.20 } & 1.15e+01 & 2.71e+00 \\
 $64 \times 64$  & { 1.64 } & 2.58e-03 & 1.61e-03 \\ \hdashline 
 $128 \times 128$  & { 1.03 } & 9.79e-05 & 2.41e-06 \\ 
 $256 \times 256$  & { 1.00 } & 5.95e-06 & 7.76e-09 \\ 
 $512 \times 512$  & \multicolumn{3}{c|}{out of memory} \\ 
\hline
\end{tabular}
\caption{Efficiency index and components of the majorant in Example~\ref{ex:sin_6pix_3piy_NTA}, $\hat{V}_{h} = \caS^{2,2}_{h}$.\label{nex_NTA}}
\end{center}\end{table}
The computed efficiency indices are presented in Table~\ref{nex_NTA}. The dashed lines correspond to criterion \eqref{e_reli} being fulfilled with $C_\oplus = 5$. We see that the proposed estimator is robust with respect to the non-trivial PDE coefficient matrix $A$, and its performance is asymptotically similar to the case with the matrix $A$ being identity
\footnote{Some deviation could be attributed to the fact that we used same number of quadrature points for the evaluation of the matrices in both the cases, which is not sufficient when the PDE coefficients are of exponential form.}.
Also, the presentation of the timings is again omitted since the overall behavior is as in Example~\ref{ex:sin_6pix_3piy}.
In the next example, we consider a domain with a curved boundary (requiring a NURBS mapping for exact representation) and a problem whose solution has sharp peaks.
\begin{example}\label{ex:q_annulus}
\textbf{Domain with curved boundary:}
Consider the domain of a quarter annulus. In polar coordinates, $\Omega$ is defined by $(r,\phi) \in (1,2) \times (0,\tfrac{\pi}{2})$. The circular parts of the domain boundary are represented exactly by the NURBS geometry mapping of degree 2, i.e., we have $p=q=2$. We set $A = I$, and we prescribe the exact solution\[ u = (r-1)(r-2)\phi (\phi-\tfrac{\pi}{2}) e^{-\alpha (r \cos\phi-1)^2}. \]
We test our method with two values of $\alpha$, namely,
\begin{align*}
\text{Example~\ref{ex:q_annulus}.a:} \quad \alpha = 20, \qquad
\text{Example~\ref{ex:q_annulus}.b:} \quad \alpha = 50.
\end{align*}
In both examples, this function has zero Dirichlet boundary values and a peak at $x = 1$, the sharpness of which is determined by the value of $\alpha$. The exact solutions are depicted in Figure~\ref{fig_nexQAex}.
\begin{figure}[!ht]\centering
\psfrag{z0}[tr][tr]{\scriptsize 0}%
\psfrag{z1}[tr][tr]{\scriptsize }%
\psfrag{z2}[br][br]{\scriptsize 0.2}%
\psfrag{y0}[tr][tr]{\scriptsize 0}%
\psfrag{y1}[tr][tr]{\scriptsize 1}%
\psfrag{y2}[tr][tr]{\scriptsize 2}%
\psfrag{x0}[tl][tl]{\scriptsize 0}%
\psfrag{x1}[tl][tl]{\scriptsize 1}%
\psfrag{x2}[tr][tr]{\scriptsize 2}%
\subfigure[Example~\ref{ex:q_annulus}.a ($\alpha = 20)$.\label{fig_nexQAex_a}]{\includegraphics[height=4cm,width=5cm]{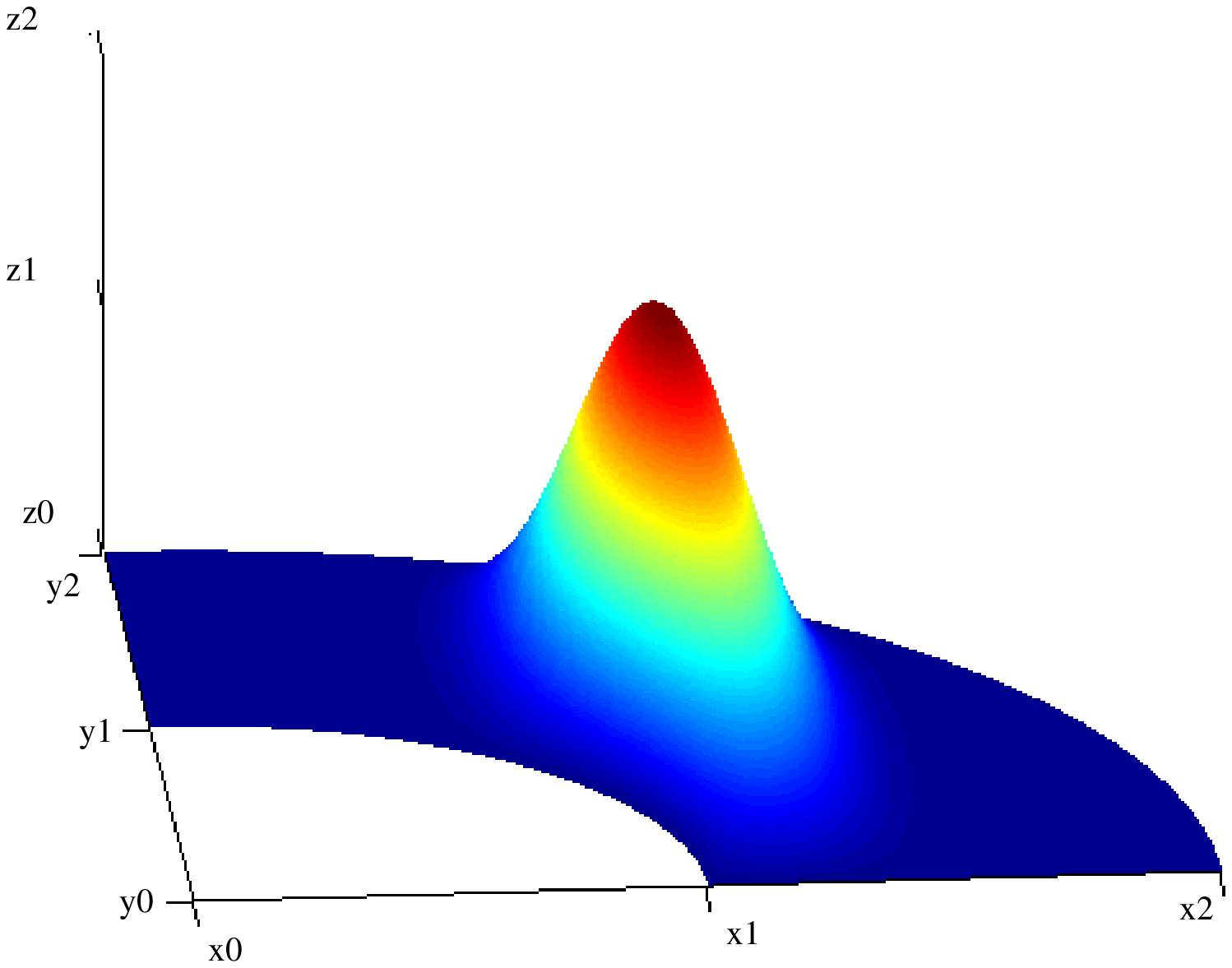}} \qquad
\subfigure[Example~\ref{ex:q_annulus}.b ($\alpha = 50)$.\label{fig_nexQAex_b}]{\includegraphics[height=4cm,width=5cm]{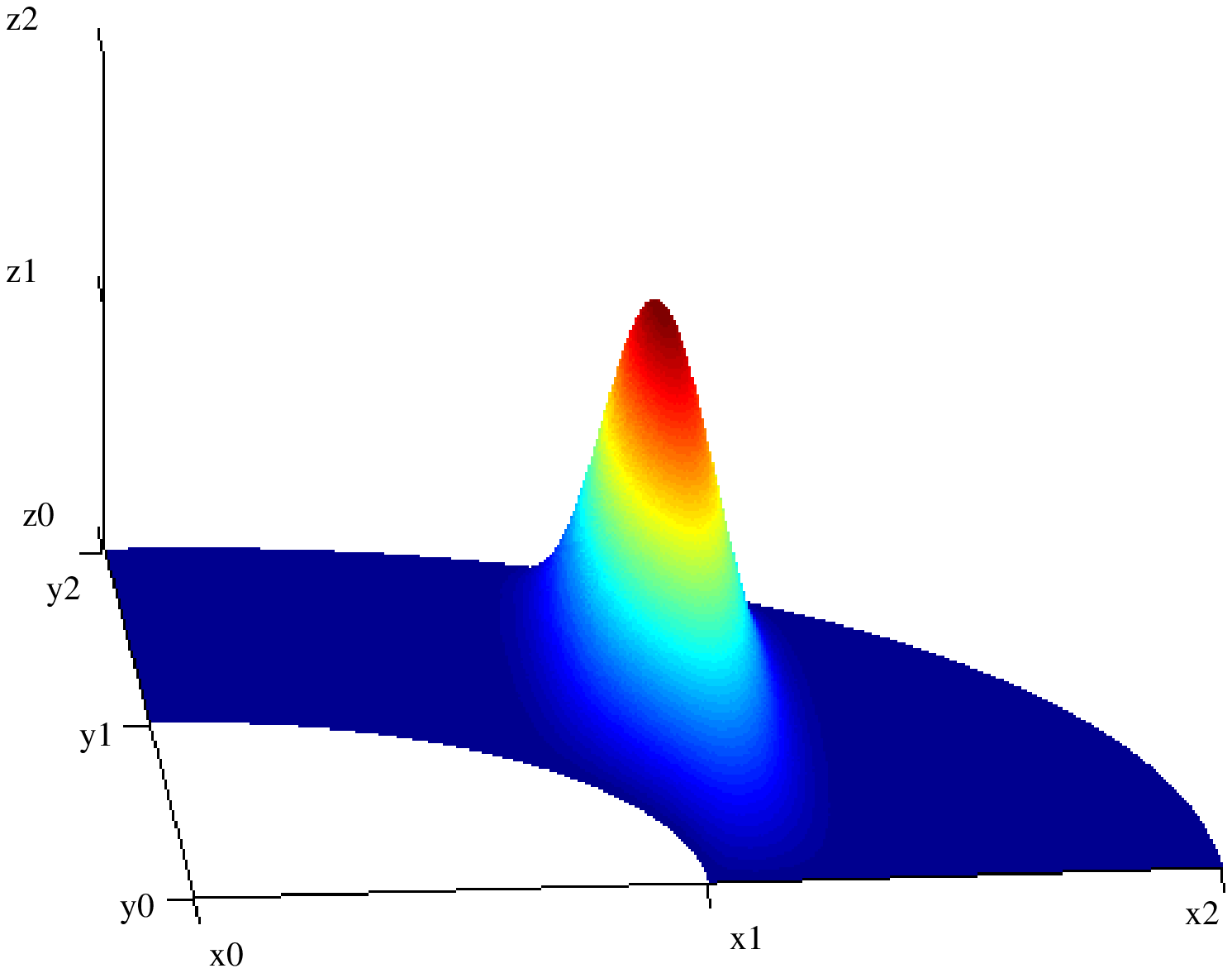}}\\
\caption{Exact solutions $u$ on $\Omega$, Example~\ref{ex:q_annulus}.\label{fig_nexQAex}}
\end{figure}
\end{example}
\begin{table}[!ht]\begin{center}
\begin{tabular}{|c|r|rr|}\hline
mesh-size & \multicolumn{1}{|c|}{$I_{\text{eff}}$}& \multicolumn{1}{|c}{$a_1 B_1$}
& \multicolumn{1}{c|}{$a_2 B_2$}
\\
\hline
\multicolumn{4}{c}{Case~\ref{ex:case_1}}\\
\hline
 $16 \times 8$  & { 1.83 } & 9.98e-04 & 3.59e-04 \\
 $32 \times 16$  & { 1.29 } & 2.08e-05 & 6.51e-06 \\\hdashline
 $64 \times 32$  & { 1.13 } & 1.04e-06 & 1.44e-07 \\
 $128 \times 64$  & { 1.07 } & 5.95e-08 & 4.00e-09 \\
 $256 \times 128$  & { 1.03 } & 3.58e-09 & 1.20e-10 \\
 $512 \times 256$  & { 1.02 } & 2.20e-10 & 3.67e-12 \\
\hline
\multicolumn{4}{c}{Case~\ref{ex:case_2}}\\
\hline
 $16 \times 8$  & { 13.99 } & 4.44e-02 & 3.51e-02  \\
 $32 \times 16$  & { 4.17 } & 2.00e-04 & 8.43e-05  \\
 $64 \times 32$  & { 1.31 } & 1.20e-06 & 3.66e-07  \\\hdashline
 $128 \times 64$  & { 1.06 } & 5.91e-08 & 3.36e-09  \\
 $256 \times 128$  & { 1.01 } & 3.51e-09 & 4.60e-11 \\
 $512 \times 256$  & { 1.00 } & 2.17e-10 & 6.96e-13 \\
\hline
\multicolumn{4}{c}{Case~\ref{ex:case_3}}\\
\hline
 $16 \times 8$  & { 24.87 } & 1.09e-01 & 1.42e-01  \\
 $32 \times 16$  & { 56.02 } & 2.92e-02 & 2.22e-02 \\
 $64 \times 32$  & { 10.42 } & 7.81e-05 & 2.16e-05 \\\hdashline
 $128 \times 64$  & { 1.11 } & 6.21e-08 & 6.61e-09  \\
 $256 \times 128$  & { 1.00 } & 3.49e-09 & 1.02e-11 \\
 $512 \times 256$  & { 1.00 } & 2.17e-10 & 3.27e-14 \\
\hline
\end{tabular}
\caption{Efficiency index and components of the majorant in Example~\ref{ex:q_annulus}.a ($\alpha = 20$), $\hat{V}_{h} = \caS^{2,2}_{h}$.\label{nexQAeff_a}}
\end{center}\end{table}
\setlength{\picw}{3.25cm}
\begin{figure}[!ht]
\centering
\subfigure[Exact, mesh $32\times 16$.]{\includegraphics[height=\picw,width=\picw]{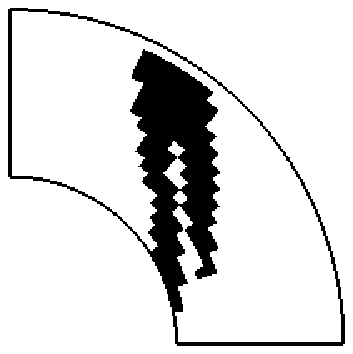}} \qquad
\subfigure[Exact, mesh $64\times 32$.]{\includegraphics[height=\picw,width=\picw]{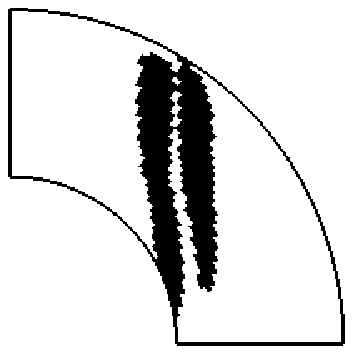}} \qquad
\subfigure[Exact, mesh $128\times 64$.]{\includegraphics[height=\picw,width=\picw]{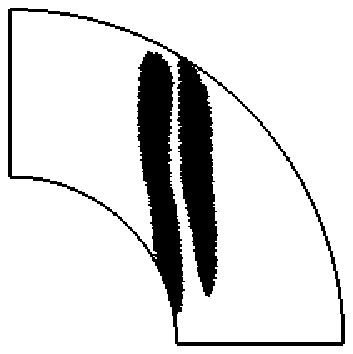}} \\
\subfigure[Case~\ref{ex:case_1}, mesh $32\times 16$.]{\includegraphics[height=\picw,width=\picw]{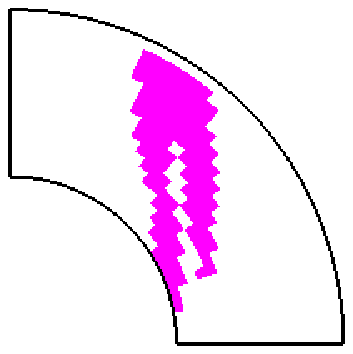}} \qquad
\subfigure[Case~\ref{ex:case_2}, mesh $64\times 32$.]{\includegraphics[height=\picw,width=\picw]{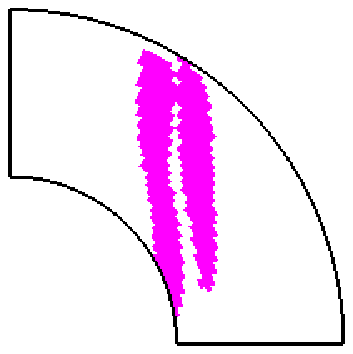}} \qquad
\subfigure[Case~\ref{ex:case_3}, mesh $128\times 64$.]{\includegraphics[height=\picw,width=\picw]{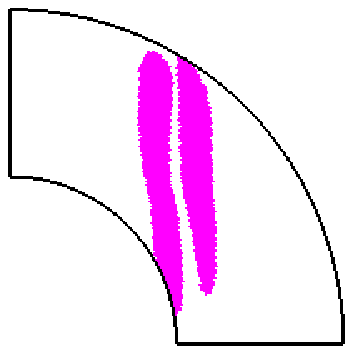}}
\caption{Marked cells with $\psi = 20\%$ in Example~\ref{ex:q_annulus}.a ($\alpha=20$), $\hat{V}_{h} = \caS^{2,2}_{h}$.\label{fig_nexQAmark_a}}
\end{figure}
\begin{table}[!ht]\begin{center}
\begin{tabular}{|c|r|rr|}\hline
 mesh-size & \multicolumn{1}{|c|}{$I_{\text{eff}}$}& \multicolumn{1}{|c}{$a_1 B_1$}
& \multicolumn{1}{c|}{$a_2 B_2$}
\\
\hline
\multicolumn{4}{c}{Case~\ref{ex:case_1}}\\
\hline 
 $16 \times 8$  & { 3.02 } & 2.94e-02 & 1.78e-02 \\
 $32 \times 16$  & { 1.92 } & 3.57e-04 & 1.83e-04 \\
 $64 \times 32$  & { 1.34 } & 9.15e-06 & 3.22e-06 \\\hdashline
 $128 \times 64$  & { 1.16 } & 4.67e-07 & 7.56e-08 \\
 $256 \times 128$  & { 1.08 } & 2.67e-08 & 2.12e-09 \\
 $512 \times 256$  & { 1.04 } & 1.60e-09 & 6.32e-11 \\
\hline
\multicolumn{4}{c}{Case~\ref{ex:case_2}}\\
\hline
 $16 \times 8$  & { 13.84 } & 3.45e-01 & 6.49e-01 \\ 
 $32 \times 16$  & { 16.76 } & 2.58e-02 & 1.53e-02 \\ 
 $64 \times 32$  & { 3.16 } & 4.10e-05 & 2.80e-05 \\ 
 $128 \times 64$  & { 1.25 } & 5.04e-07 & 1.24e-07 \\\hdashline 
 $256 \times 128$  & { 1.05 } & 2.61e-08 & 1.33e-09 \\ 
 $512 \times 256$  & { 1.01 } & 1.56e-09 & 1.89e-11 \\ 
\hline
\multicolumn{4}{c}{Case~\ref{ex:case_3}}\\
\hline
 $16 \times 8$  & { 17.20 } & 4.24e-01 & 1.11e+00 \\ 
 $32 \times 16$  & { 76.95 } & 3.24e-01 & 5.41e-01 \\ 
 $64 \times 32$  & { 83.72 } & 3.02e-02 & 1.83e-02 \\ 
 $128 \times 64$  & { 4.19 } & 4.64e-06 & 2.44e-06 \\\hdashline 
 $256 \times 128$  & { 1.04 } & 2.59e-08 & 1.02e-09 \\ 
 $512 \times 256$  & { 1.00 } & 1.55e-09 & 2.22e-12 \\ 
\hline
\end{tabular}
\caption{Efficiency index and components of the majorant in Example~\ref{ex:q_annulus}.b ($\alpha = 50$), $\hat{V}_{h} = \caS^{2,2}_{h}$.\label{nexQAeff_b}}
\end{center}\end{table}
\begin{figure}[!ht]
\centering
\subfigure[Exact, mesh $64\times 32$.]{\includegraphics[height=\picw,width=\picw]{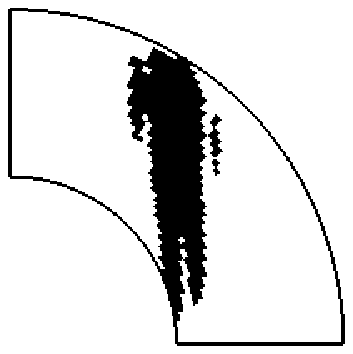}} \qquad
\subfigure[Exact, mesh $128\times 64$.]{\includegraphics[height=\picw,width=\picw]{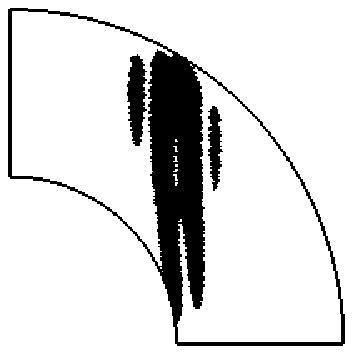}} \qquad
\subfigure[Exact, mesh $256\times 128$.]{\includegraphics[height=\picw,width=\picw]{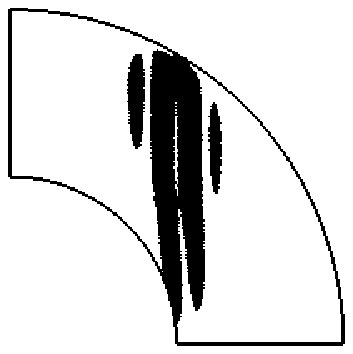}} \\
\subfigure[Case~\ref{ex:case_1}, mesh $64\times 32$.]{\includegraphics[height=\picw,width=\picw]{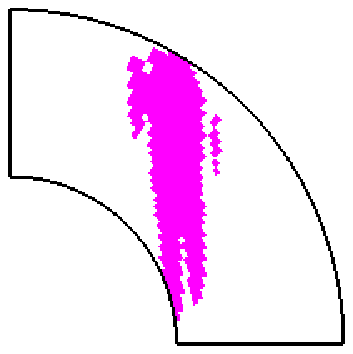}} \qquad
\subfigure[Case~\ref{ex:case_2}, mesh $128\times 64$.]{\includegraphics[height=\picw,width=\picw]{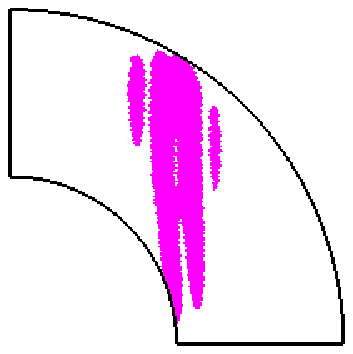}} \qquad
\subfigure[Case~\ref{ex:case_3}, mesh $256\times 128$.]{\includegraphics[height=\picw,width=\picw]{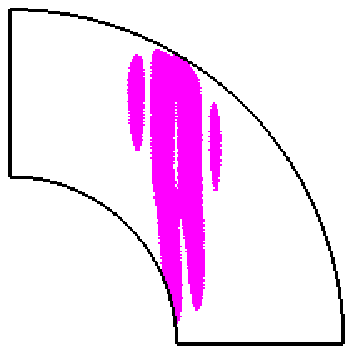}}
\caption{Marked cells with $\psi = 20\%$ in Example~\ref{ex:q_annulus}.b ($\alpha=50$), $\hat{V}_{h} = \caS^{2,2}_{h}$.\label{fig_nexQAmark_b}}
\end{figure}
In Tables~\ref{nexQAeff_a} and \ref{nexQAeff_b}, the efficiency index $I_{\text{eff}}$ and the magnitudes of $a_1 B_1$ and $a_2 B_2$ are presented for both the cases of $\alpha$. The dashed lines indicate the mesh-size after which criterion \eqref{e_reli} with $C_\oplus = 5$ is fulfilled. The distribution of the marked cells is depicted in Figures~\ref{fig_nexQAmark_a} and \ref{fig_nexQAmark_b}. As before, we observe that the error distribution is represented correctly if the criterion \eqref{e_reli} is fulfilled with $C_\oplus = 5$.
When comparing Tables~\ref{nexQAeff_a} and \ref{nexQAeff_b}, as well as  Figures~\ref{fig_nexQAmark_a} and \ref{fig_nexQAmark_b}, we notice that the more aggressive the mesh coarsening, and sharper the peak, the more refinements are needed before criterion \eqref{e_reli} is fulfilled and the error distribution is captured correctly. 
Since the timings in Example~\ref{ex:q_annulus}.a and Example~\ref{ex:q_annulus}.b show the same behavior as in Example~\ref{ex:sin_6pix_3piy}, both regarding assembling-time and solving-time, we omit the presentation of these numbers. Clearly, Case~\ref{ex:case_3} outperforms Cases~\ref{ex:case_1} and~\ref{ex:case_2} in terms of cost-efficiency.
In the next example, we test the proposed estimator in a basic adaptive refinement scheme.

\begin{example}\label{ex:adap_ref}
\textbf{Adaptive Refinement:}
The exact solution for this example is given by
\[
u = ( x^2-x)(y^2-y)  e^{-100  |(x,y)- (0.8,0.05)|^2 -100 |(x,y)- (0.8,0.95)|^2}.
\]
The computational domain is again the unit square $\Omega = (0,1)^2$, and is represented by B-splines of degree $p=q=2$. The function $u$, which is illustrated in Figure~\ref{fig_nexAda_exsol}, has zero Dirichlet boundary values and has two peaks at the coordinates $(0.8,0.05)$ and $(0.8,0.95)$.
\end{example}
In this example, we test a very basic adaptive refinement procedure using tensor-product B-splines. The discussion of isogeometric local refinement schemes  is out of the scope of this paper (see Section~\ref{sec_intro} for an overview on local refinement methods).
\begin{figure}[!ht]\centering
\psfrag{r1}[tl][tl]{\footnotesize $1$}%
\psfrag{r0}[tl][tl]{\footnotesize $0$}%
\psfrag{l1}[tr][tr]{\footnotesize $1$}%
\psfrag{l0}[tr][tr]{\footnotesize $0$}%
\includegraphics[height=5cm,width=5cm]{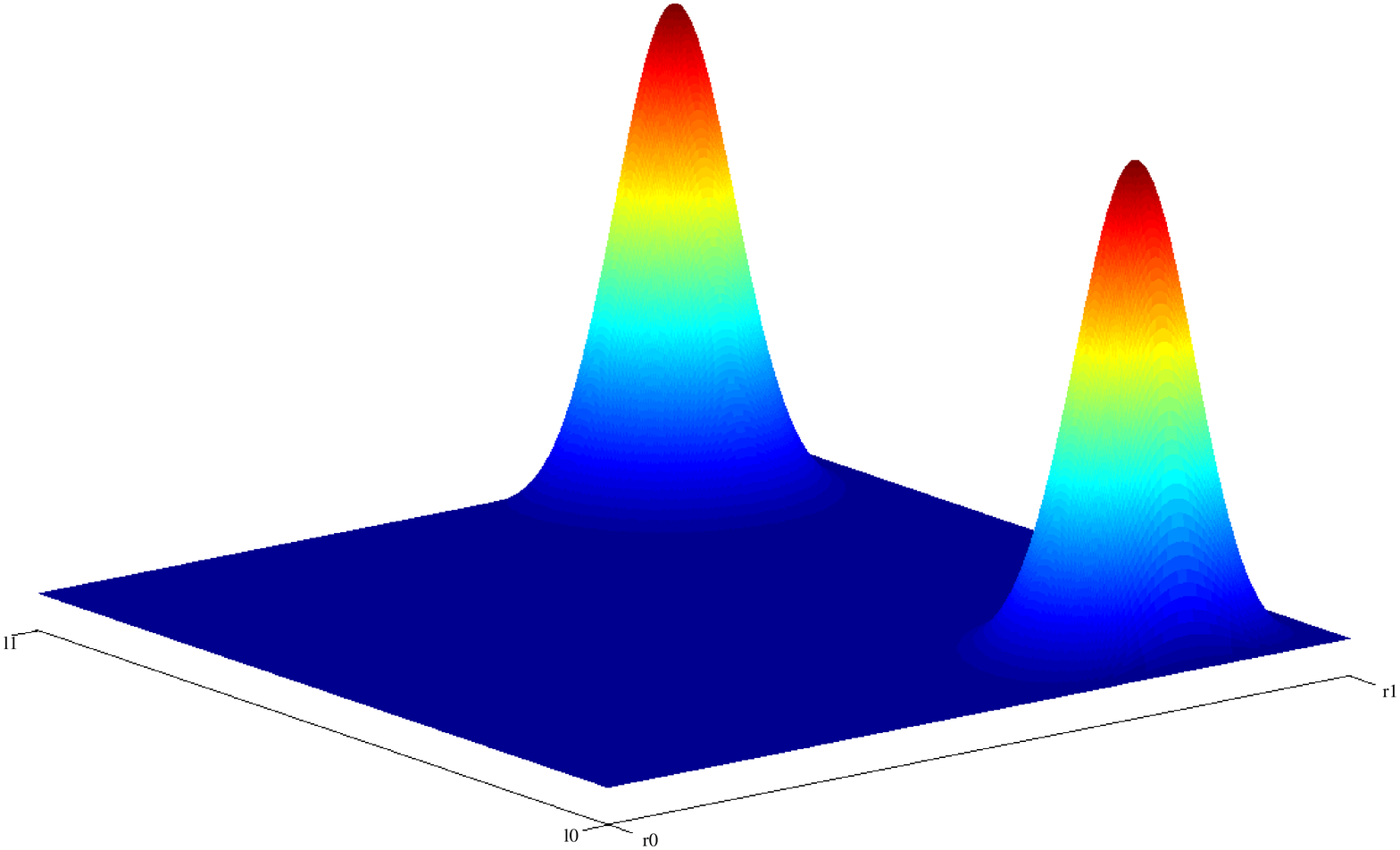}
\caption{Exact solution, Example~\ref{ex:adap_ref}.\label{fig_nexAda_exsol}}
\end{figure}

\begin{figure}[!ht]\centering
\includegraphics[scale=0.65]{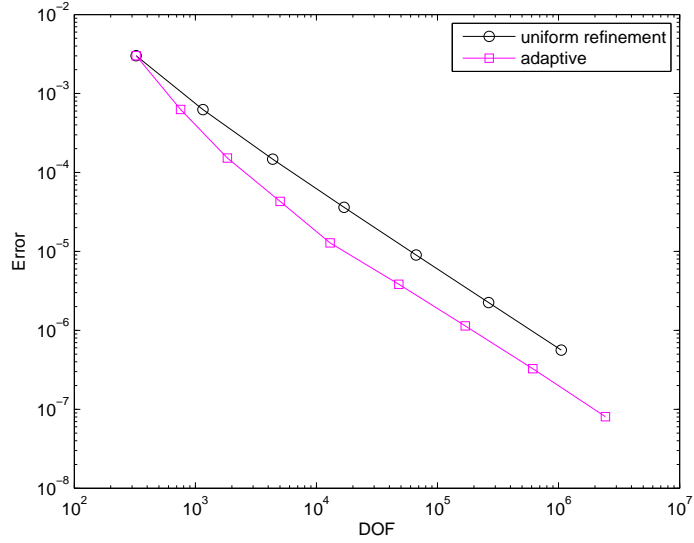}
\caption{Error convergence, Example~\ref{ex:adap_ref}, Cases as in Table~\ref{nexAda_Ieff}.\label{fig_nexAda_errconv}}
\end{figure}

\begin{table}[!ht]\begin{center}
\begin{tabular}{|c|r|rr|c|}\hline
 mesh-size & \multicolumn{1}{|c|}{$I_{\text{eff}}$}& \multicolumn{1}{|c}{$a_1 B_1$}
& \multicolumn{1}{c|}{$a_2 B_2$}
& Case \\
\hline
$16 \times 16$  & { 3.77 } & 9.39e-05 & 3.49e-05 & \ref{ex:case_1}\\
$25 \times 26$  & { 2.06 } & 8.62e-07 & 8.11e-07 & \ref{ex:case_1}\\
$38 \times 44$  & { 1.69 } & 4.30e-08 & 2.35e-08 & \ref{ex:case_1}\\
$64 \times 74$  & { 1.47 } & 2.79e-09 & 1.19e-09 & \ref{ex:case_1}\\
$92 \times 136$  & { 2.82 } & 8.19e-10 & 4.87e-10 & \ref{ex:case_2}\\
$184 \times 256$  & { 1.30 } & 2.05e-11 & 4.55e-12 & \ref{ex:case_2}\\
$341 \times 492$  & { 1.11 } & 1.45e-12 & 1.47e-13 & \ref{ex:case_2}\\
$652 \times 934$  & { 1.84 } & 2.55e-13 & 1.07e-13 & \ref{ex:case_3}\\
$1304 \times 1868$& { 1.09 } & 7.40e-15 & 3.63e-16 & \ref{ex:case_3}\\
\hline
\end{tabular}
\caption{Efficiency index, components of the majorant and applied cases in Example~\ref{ex:adap_ref}, $\hat{V}_{h} = \caS^{2,2}_{h}$. \label{nexAda_Ieff}}
\end{center}\end{table}

\begin{figure}[!ht]\centering%
\subfigure[Mesh 4.\label{fig_nexAda_mesh4}]{\includegraphics[height=4cm,width=4cm]{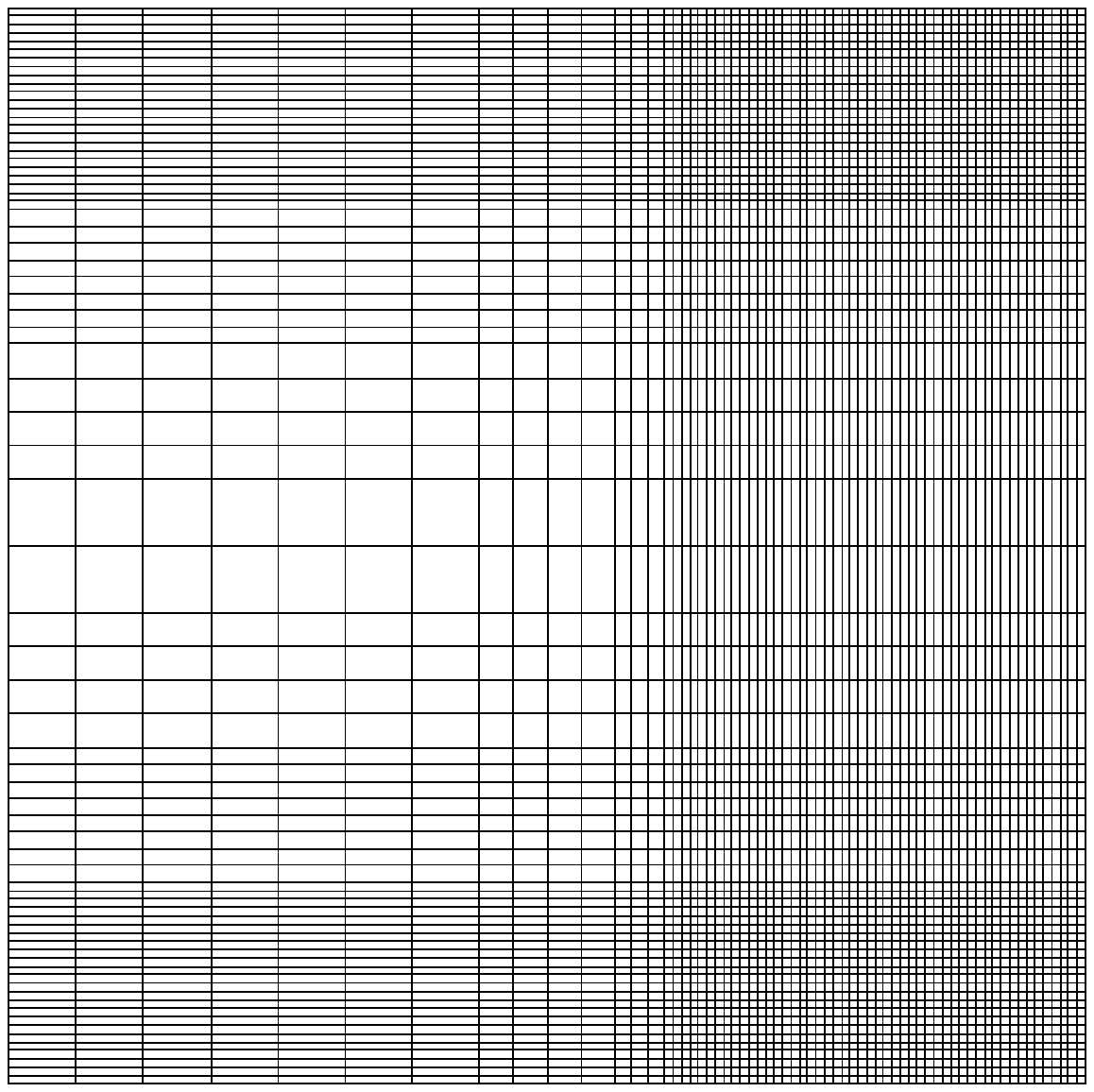}} \qquad
\subfigure[Cells marked by exact error on mesh 4.\label{fig_nexAda_err4}]{\includegraphics[height=4cm,width=4cm]{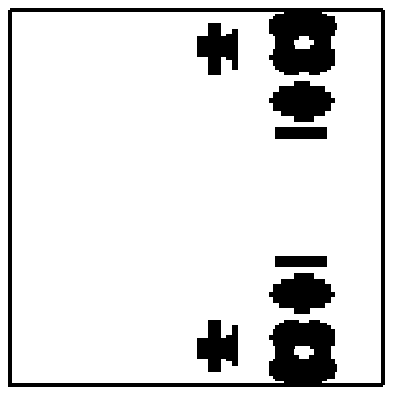}} \qquad
\subfigure[Cells marked by estimator on mesh 4.\label{fig_nexAda_est4}]{\includegraphics[height=4cm,width=4cm]{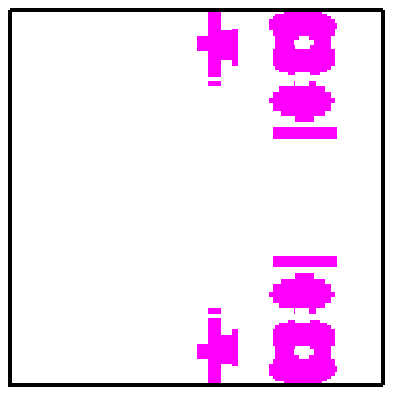}} \\
\subfigure[Mesh 7.\label{fig_nexAda_mesh7}]{\includegraphics[height=4cm,width=4cm]{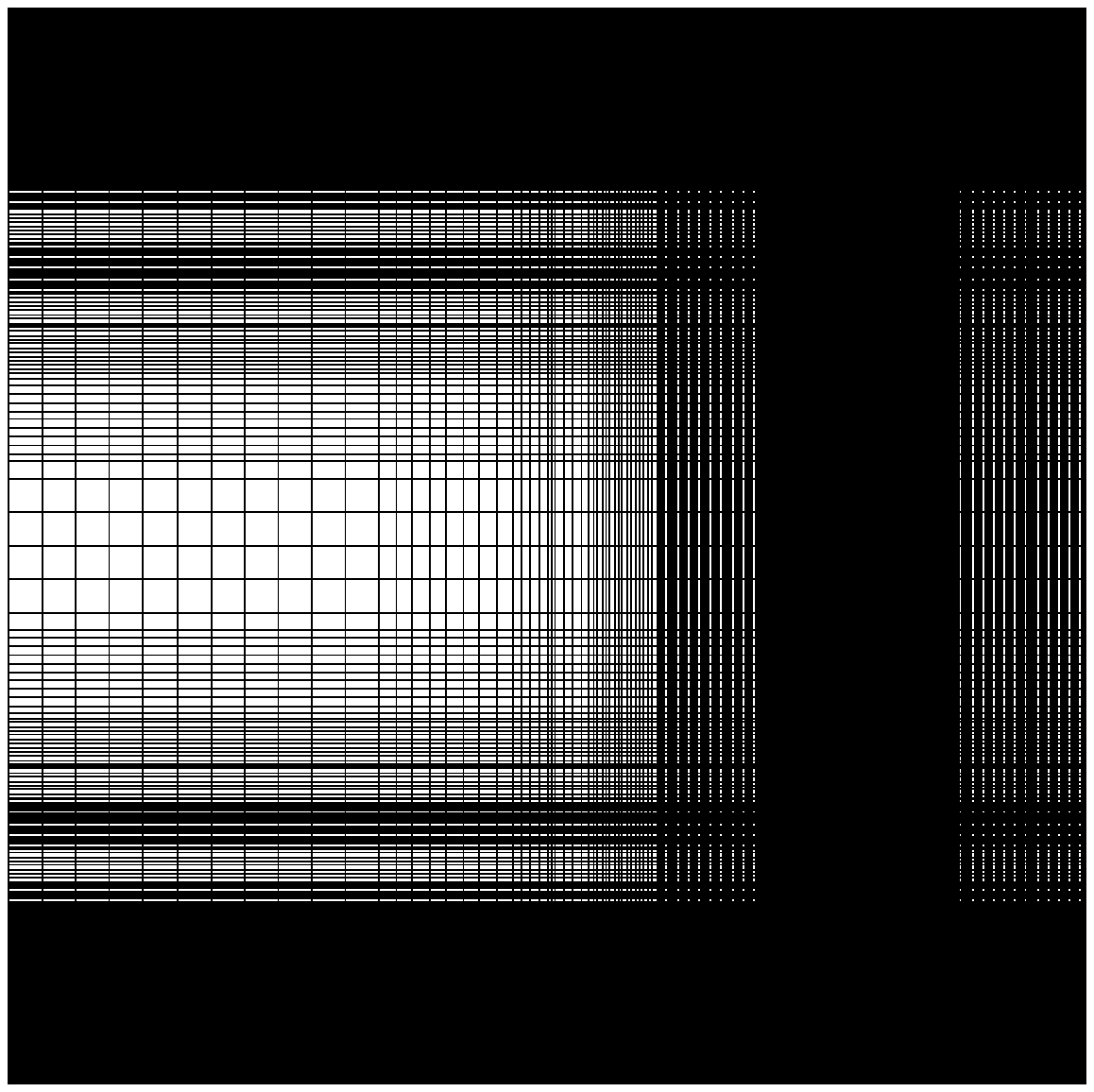}} \qquad
\subfigure[Cells marked by exact error on mesh 7.\label{fig_nexAda_err7}]{\includegraphics[height=4cm,width=4cm]{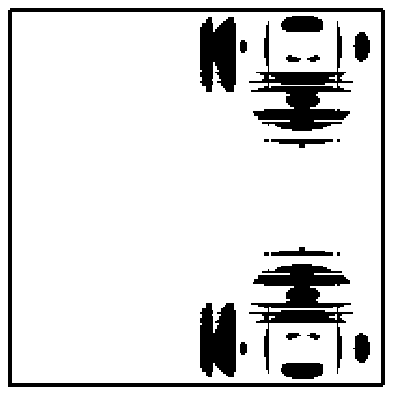}} \qquad
\subfigure[Cells marked by estimator on mesh 7.\label{fig_nexAda_est7}]{\includegraphics[height=4cm,width=4cm]{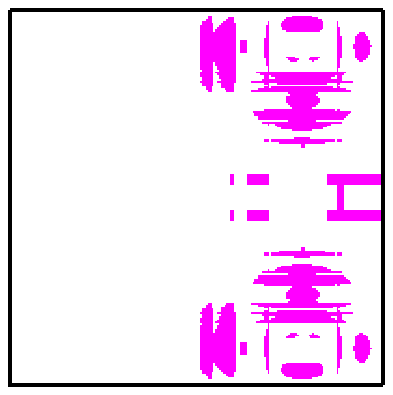}} \\
\subfigure[Mesh 9.\label{fig_nexAda_mesh9}]{\includegraphics[height=4cm,width=4cm]{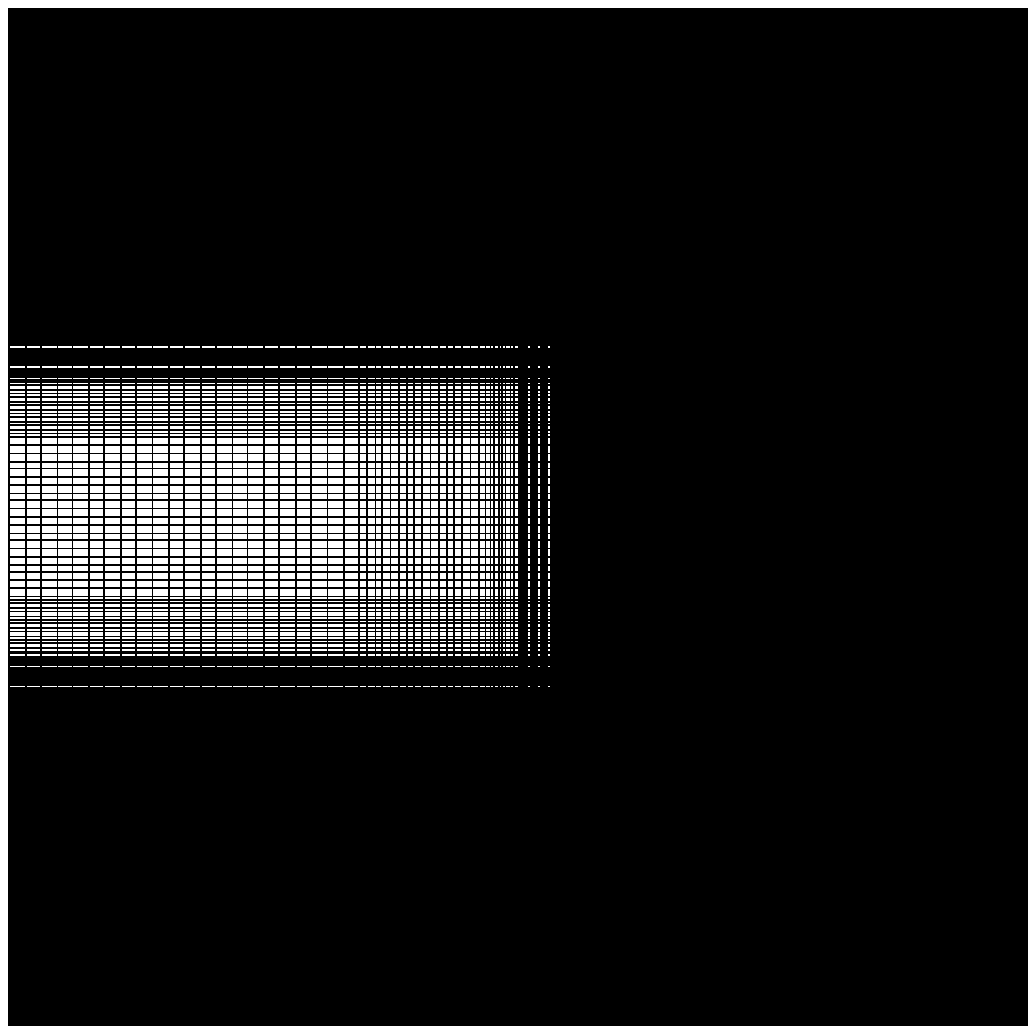}} \qquad
\subfigure[Cells marked by exact error on mesh 9.\label{fig_nexAda_err9}]{\includegraphics[height=4cm,width=4cm]{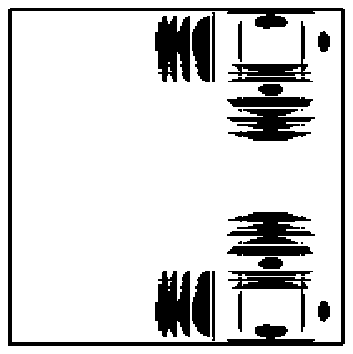}} \qquad
\subfigure[Cells marked by estimator on mesh 9.\label{fig_nexAda_est9}]{\includegraphics[height=4cm,width=4cm]{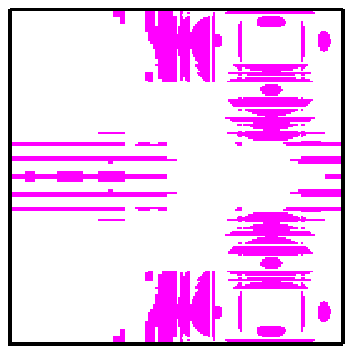}} 
\caption{Meshes and marked cells in Example~\ref{ex:adap_ref}, Cases as in Table~\ref{nexAda_Ieff}, $\psi = 25\%$, $\hat{V}_{h} = \caS^{2,2}_{h}$.\label{fig_nexAda_meshes}}
\end{figure}
We apply adaptive refinement based on a marking with $\psi = 25\%$, starting on an initial mesh $16\times 16$. On the first four steps, we apply Case~\ref{ex:case_1}, then Case~\ref{ex:case_2} on the next three steps, and thereafter Case~\ref{ex:case_3}. The efficiency indices and the applied cases are shown in Table~\ref{nexAda_Ieff}. In Figure~\ref{fig_nexAda_meshes}, the meshes and the marked cells are shown for steps~4, 7, and~9. Clearly, the correct areas of the domain are identified and marked for refinement. 
Since the solution of the problem is sufficiently regular, the error plots in Figure~\ref{fig_nexAda_errconv} show that the adaptive refinement converges with the same rate as the uniform refinement, but with a better constant. However, due to the tensor-product structure of the mesh, many superfluous DOFs are inserted outside of the marked areas, which worsens the rate of convergence for given total DOFs. 
Next, we consider a classical example for a posteriori error estimation and adaptive refinement studies.
\begin{example}\label{ex:l_shaped}
\textbf{L-shaped domain:}
We consider the Laplace equation
\begin{equation}
\Delta u = 0 \label{nexLSlaplace}
\end{equation}
with Dirichlet boundary conditions on the L-shaped domain $\Omega = (-1,1)^2 \backslash [0,1]^2$. In this example, we use a bilinear geometry mapping, i.e., $p=q=1$. The function
\begin{equation*}
u(r,\phi) = r^{\frac{2}{3}} \sin( (2\phi-\pi)/3 )
\end{equation*}
solves \eqref{nexLSlaplace} and is used to prescribe Dirichlet boundary conditions. The solution has a singularity at the re-entrant corner at $(0,0)$.
\end{example}
We compare uniform refinement and adaptive refinement in the tensor-product setting. In this example we set $\psi = 10\%$, and to avoid the pollution near the singularity, we only use Case~\ref{ex:case_1} in the majorant computations.
\begin{table}[!ht]\begin{center}
\begin{tabular}{|c|r|rr|c}\hline
 mesh-size & \multicolumn{1}{|c|}{$I_{\text{eff}}$}& \multicolumn{1}{|c}{$a_1 B_1$}
& \multicolumn{1}{c|}{$a_2 B_2$}
\\
\hline
$16 \times 8$  & { 1.1785 } & 5.67e-02 & 1.71e-03 \\
$32 \times 16$  & { 1.1401 } & 3.44e-02 & 8.98e-04 \\
$64 \times 32$  & { 1.1116 } & 2.09e-02 & 4.72e-04 \\
$128 \times 64$  & { 1.0898 } & 1.28e-02 & 2.49e-04 \\
$256 \times 128$  & { 1.0729 } & 7.87e-03 & 1.32e-04 \\
$512 \times 256$  & { 1.0593 } & 4.86e-03 & 7.01e-05 \\
$1024 \times 512$  & { 1.0485 } & 3.01e-03 & 3.73e-05 \\
\hline
\end{tabular}
\caption{Efficiency index and components of the majorant in Example~\ref{ex:l_shaped}, Case~\ref{ex:case_1}, uniform refinement, $\hat{V}_{h} = \caS^{1,1}_{h}, \hat{Y}_h = \caS^{2,2}_{h} \otimes \caS^{2,2}_{h}$.\label{nexLSeff}}
\end{center}\end{table}
\begin{table}[!ht]\begin{center}
\begin{tabular}{|c|r|rr|c}\hline
 mesh-size & \multicolumn{1}{|c|}{$I_{\text{eff}}$}& \multicolumn{1}{|c}{$a_1 B_1$}
& \multicolumn{1}{c|}{$a_2 B_2$}\\
\hline
$16 \times 8$  & { 1.1785 } & 5.67e-02 & 1.71e-03 \\
$22 \times 11$  & { 1.1839 } & 2.68e-02 & 8.66e-04 \\
$30 \times 16$  & { 1.1749 } & 1.37e-02 & 4.32e-04 \\
$39 \times 23$  & { 1.1622 } & 7.22e-03 & 2.24e-04 \\
$55 \times 37$  & { 1.1635 } & 3.52e-03 & 1.10e-04 \\
$87 \times 60$  & { 1.1634 } & 1.75e-03 & 5.41e-05 \\
$133 \times 101$  & { 1.1525 } & 9.25e-04 & 2.69e-05 \\
\hline
\end{tabular}
\caption{Efficiency index and components of the majorant in Example~\ref{ex:l_shaped}, Case~\ref{ex:case_1}, adaptive refinement, $\hat{V}_{h} = \caS^{1,1}_{h}, \hat{Y}_h = \caS^{2,2}_{h} \otimes \caS^{2,2}_{h}$.\label{nexLSeff_ada}}
\end{center}\end{table}
The magnitudes of the components $a_1 B_1$ and $a_2B_2$, which are presented in Table~\ref{nexLSeff} for uniform refinement, and in Table~\ref{nexLSeff_ada} for adaptive refinement, show that the criterion \eqref{e_reli} with $C_\oplus = 5$ is fulfilled on all the considered meshes. 
The error plots presented in Figure~\ref{fig_Lada_errconv} show the expected faster convergence on the adaptively refined mesh, even though we are only using tensor-product splines. In Figure~\ref{fig_Lada_meshes}, meshes and marked cells are shown for steps 2 and 6, again indicating that the error indicator correctly identifies the corner singularity. 
\begin{figure}[!ht]\centering
\includegraphics[scale=0.65]{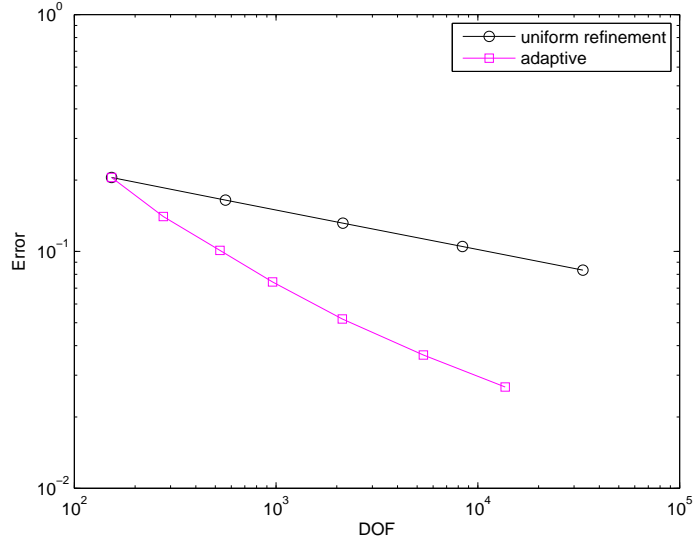}
\caption{Error convergence, Example~\ref{ex:l_shaped}, Case~\ref{ex:case_1}.\label{fig_Lada_errconv}}
\end{figure}

\begin{figure}[!ht]\centering%
\subfigure[Mesh after 2 refinements.\label{fig_Lada_mesh2}]{\includegraphics[height=4cm,width=4cm]{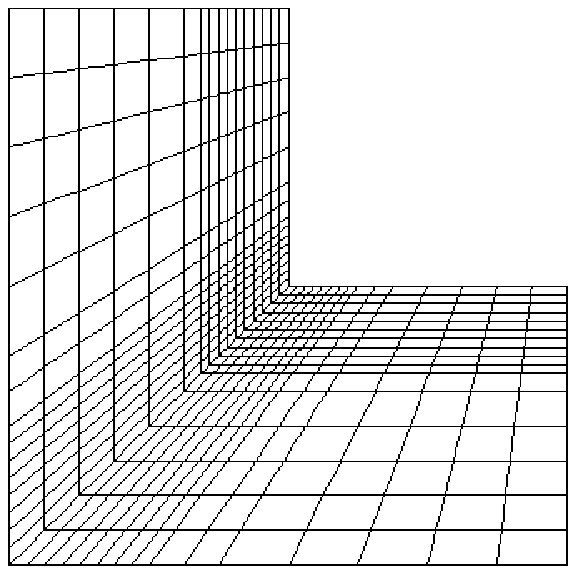}} 
\qquad
\subfigure[Cells marked by exact error on mesh 2.\label{fig_Lada_err2}]{\includegraphics[height=4cm,width=4cm]{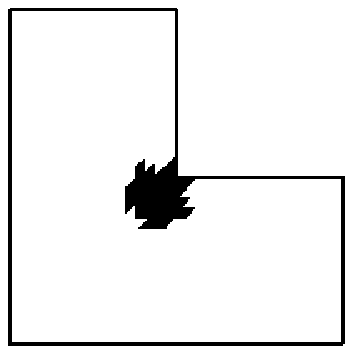}}
\qquad
\subfigure[Cells marked by estimator on mesh 2.\label{fig_Lada_est2}]{\includegraphics[height=4cm,width=4cm]{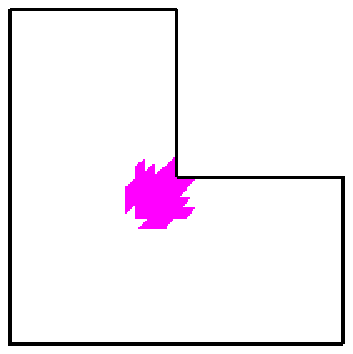}} 
\\
\subfigure[Mesh after 6 refinements.\label{fig_Lada_mesh6}]{\includegraphics[height=4cm,width=4cm]{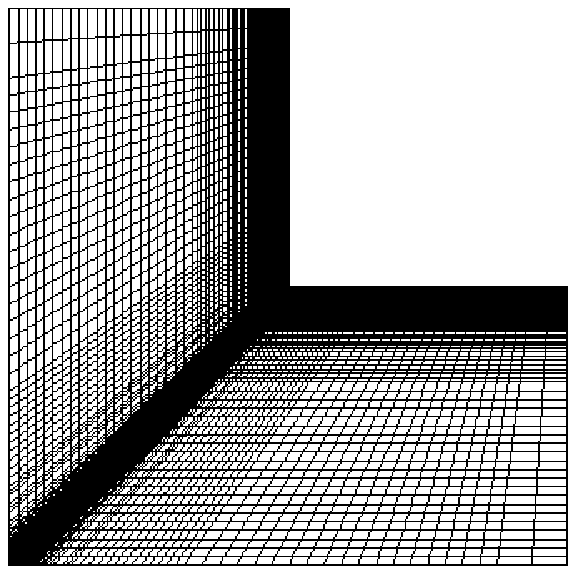}} 
\qquad
\subfigure[Cells marked by exact error on mesh 6.\label{fig_Lada_err6}]{\includegraphics[height=4cm,width=4cm]{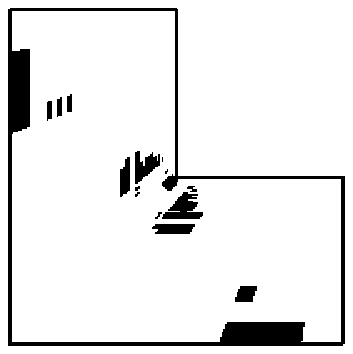}}
\qquad
\subfigure[Cells marked by estimator on mesh 6.\label{fig_Lada_est6}]{\includegraphics[height=4cm,width=4cm]{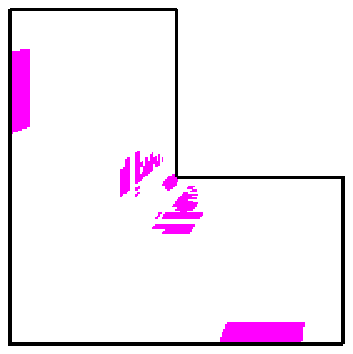}} 
\caption{Meshes and marked cells in Example~\ref{ex:l_shaped}, Case~\ref{ex:case_1}, $\psi = 10\%$, $\hat{V}_{h} = \caS^{1,1}_{h}, \hat{Y}_h = \caS^{2,2}_{h} \otimes \caS^{2,2}_{h}$.\label{fig_Lada_meshes}}
\end{figure}

In our final example, we consider an advection dominated advection diffusion equation to see the performance of the estimator for sharp boundary layers.
\begin{example}\label{ex:adv_diff}
\textbf{Advection dominated advection diffusion equation:}
We consider the advection diffusion equation with Dirichlet boundary conditions on the unit square $\Omega = (0,1)^2$, with $p=q=2$, i.e.,
\begin{align*}
\begin{array}{r@{\ =\ }ll}
-\kappa \Delta u + b\cdot \nabla u & 0 \quad &\text{in~}\Omega,\\
u & u_D \quad &\text{on~} \partial \Omega,
\end{array}
\end{align*}
where
\begin{align*}
\kappa = 10^{-6}, \qquad
b = (\cos \tfrac{\pi}{3}, \sin \tfrac{\pi}{3})^T, \qquad
u_D = \left\{\begin{array}{ll}
1,\text{~if~}y=0 \\
0,\text{~else}
\end{array}\right. .
\end{align*}
\end{example}
We use the standard streamline upwind Petrov-Galerkin (SUPG) scheme for the stabilization. The stabilization parameter $\tau$ is set to $\tau(Q) = h_b(Q)/2|b|$, where $h_b(Q)$ is the diameter of the cell $Q$ in direction of the flow $b$, and $|b|$ is the magnitude of the vector $b$. For advection diffusion problems, we have to adapt the majorant. Since the principle method is the same, we refer the reader to \cite[Section~4.3.1]{Repin08_book} for a detailed discussion. In this special case, where $A = \kappa I$ with $\kappa \ll |b|$, and with constant velocity vector $b$, the majorant $M_{\oplus,\text{adv}}^{2}$ for the advection diffusion problem is given by
\[
M_{\oplus,\text{adv}}^{2} = (1+\beta) \|  A \nabla u_h - y \|_{\bar{A}}^2 + (1+\tfrac{1}{\beta}) C_{\Omega}^2 \| \dvg y + f - b\cdot\nabla u_h \|^2.
\]
The strong advection and the discontinuous boundary conditions result in sharp layers. In Figure~\ref{fig_E4_lay}, the expected positions of the layers are indicated by dashed lines. 
\begin{table}[!ht]\begin{center}
\begin{tabular}{|c|rr|}\hline
 mesh-size &  \multicolumn{1}{|c}{$a_1 B_1$}& \multicolumn{1}{c|}{$a_2 B_2$} \\
\hline
\multicolumn{3}{c}{Case~\ref{ex:case_1}}\\
\hline
 $16 \times 16$  & 1.98e-07 & 3.18e-10 \\
 $64 \times 64$  &  6.45e-07 & 1.15e-09 \\
 $256 \times 256$   & 2.28e-06 & 4.33e-09 \\
\hline
\multicolumn{3}{c}{Case~\ref{ex:case_2}}\\
\hline
 $16 \times 16$  & 1.83e-06 & 9.66e-10 \\
 $64 \times 64$   & 6.50e-06 & 3.65e-09 \\
 $256 \times 256$   & 1.86e-05 & 1.24e-08 \\
\hline
\multicolumn{3}{c}{Case~\ref{ex:case_3}}\\
\hline
 $16 \times 16$  & 3.24e-06 & 1.29e-09 \\ 
 $64 \times 64$  &  2.07e-05 & 6.52e-09 \\ 
 $256 \times 256$  & 6.86e-05 & 2.38e-08 \\ 
\hline
\end{tabular}\caption{Comparison of terms $a_1B_1$ and $a_2 B_2$ in Example~\ref{ex:adv_diff}, $\hat{V}_{h} = \caS^{2,2}_{h}$.\label{nexADeff}}
\end{center}\end{table}
\setlength{\picw}{4cm}
\begin{figure}[!ht]\centering
\psfrag{x0}{{\footnotesize $0$}}%
\psfrag{x1}{{\footnotesize $1$}}%
\psfrag{y0}{{\footnotesize $0$}}%
\psfrag{y1}{{\footnotesize $1$}}%
\psfrag{x}{{\footnotesize $x$}}%
\psfrag{y}{{\footnotesize $y$}}%
\psfrag{60deg}{{\footnotesize $60^\circ$}}%
\subfigure[Expected positions of sharp layers.\label{fig_E4_lay}]{\includegraphics[height=\picw,width=\picw]{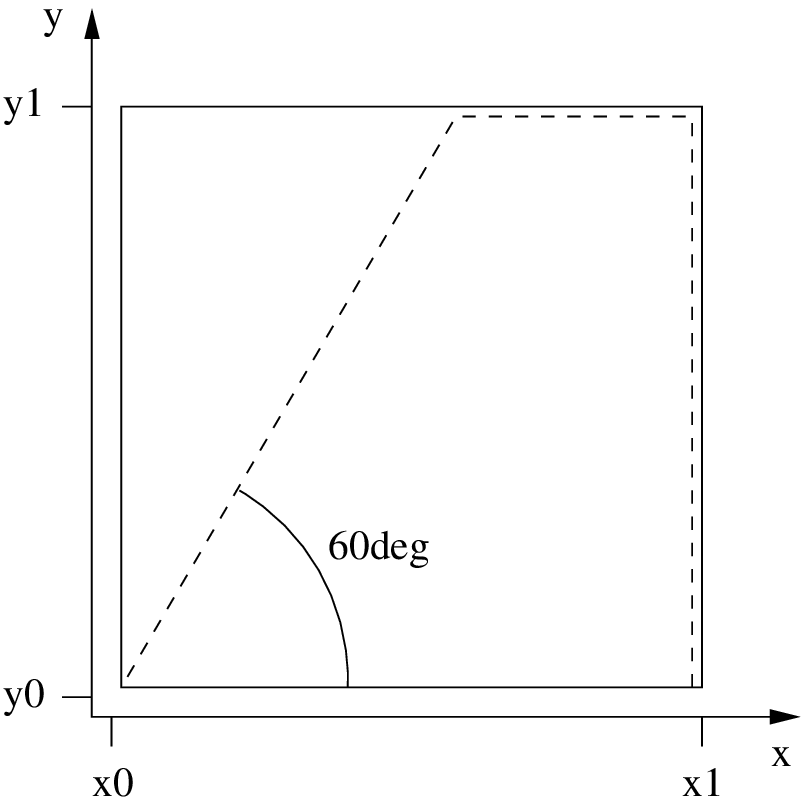}} \qquad
\subfigure[Marked cells with $\psi=20\%$, mesh-size $64\times 64$.\label{fig_E4_064}]{\includegraphics[height=\picw,width=\picw]{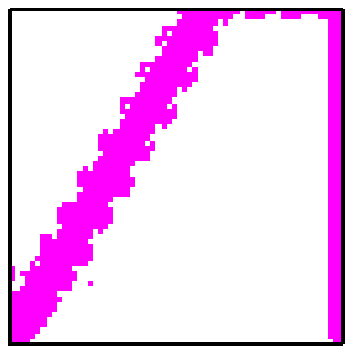}} \qquad
\subfigure[Marked cells with $\psi=10\%$, mesh-size $256\times 256$.\label{fig_E4_256}]{\includegraphics[height=\picw,width=\picw]{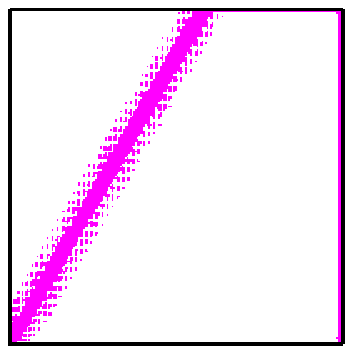}}
\caption{Expected layers and marked cells in Example~\ref{ex:adv_diff}, $\hat{V}_{h} = \caS^{2,2}_{h}, \hat{Y}_h = \caS^{6,6}_{4h} \otimes \caS^{6,6}_{4h}$.\label{fig_nexADmark}}
\end{figure}
\begin{table}[!ht]\begin{center}
\begin{tabular}{|c|rr|rrr|rrr|rrr|}\hline
mesh-size & 
\multicolumn{2}{|c|}{$\#$DOF} & 
\multicolumn{3}{|c|}{assembling-time} & 
\multicolumn{3}{|c|}{solving-time} & 
\multicolumn{3}{|c|}{sum} 
\\
& $u_h$ & $y_h$ 
& \multicolumn{1}{|c}{\emph{pde}} 
& \multicolumn{1}{c}{\emph{est}} 
& \multicolumn{1}{c|}{$\frac{\emph{est}}{\emph{pde}}$} 
& \multicolumn{1}{|c}{\emph{pde}} 
& \multicolumn{1}{c}{\emph{est}} 
& \multicolumn{1}{c|}{$\frac{\emph{est}}{\emph{pde}}$} 
& \multicolumn{1}{|c}{\emph{pde}} 
& \multicolumn{1}{c}{\emph{est}} 
& \multicolumn{1}{c|}{$\frac{\emph{est}}{\emph{pde}}$} 
\\
\hline 
\multicolumn{12}{c}{Case~\ref{ex:case_1}}\\
\hline
 $16 \times 16$ & 324 & 722 & 0.25 & 0.39 & { 1.56 } & $<$0.01 & 0.01 & { 6.38 } & 0.25 & 0.40 & { 1.59 } \\
 $64 \times 64$ & 4356 & 8978 & 3.25 & 5.32 & { 1.63 } & 0.03 & 0.26 & { 8.64 } & 3.28 & 5.58 & { 1.70 } \\
 $256 \times 256$ & 66564 & 134162 & 51.22 & 94.15 & { 1.84 } & 0.85 & 8.84 & { 10.35 } & 52.07 & 102.99 & { 1.98 } \\
\hline 
\multicolumn{12}{c}{Case~\ref{ex:case_2}}\\
\hline
 $16 \times 16$ & 324 & 288 & 0.21 & 0.14 & { 0.67 } & $<$0.01 & $<$0.01 & { 0.50 } & 0.21 & 0.14 & { 0.67 } \\
 $64 \times 64$ & 4356 & 2592 & 3.26 & 2.10 & { 0.64 } & 0.03 & 0.06 & { 2.01 } & 3.29 & 2.16 & { 0.66 } \\
 $256 \times 256$ & 66564 & 34848 & 50.83 & 35.58 & { 0.70 } & 0.85 & 2.30 & { 2.70 } & 51.68 & 37.87 & { 0.73 } \\
\hline 
\multicolumn{12}{c}{Case~\ref{ex:case_3}}\\
\hline
 $16 \times 16$ & 324 & 200 & 0.26 & 0.10 & { 0.39 } & $<$0.01 & $<$0.01 & { 0.58 } & 0.26 & 0.10 & { 0.40 } \\ 
 $64 \times 64$ & 4356 & 968 & 3.41 & 1.21 & { 0.35 } & 0.04 & 0.01 & { 0.26 } & 3.44 & 1.22 & { 0.35 } \\ 
 $256 \times 256$ & 66564 & 9800 & 52.40 & 19.83 & { 0.38 } & 1.02 & 0.91 & { 0.89 } & 53.42 & 20.74 & { 0.39 } \\ 
\hline\end{tabular}
\caption{Timings in Example~\ref{ex:adv_diff}, $\hat{V}_{h} = \caS^{2,2}_{h}$.\label{nexADtime}}
\end{center}\end{table}
The magnitudes of $a_1B_1$ and $a_2B_2$ presented in Table~\ref{nexADeff} indicate that the criterion \eqref{e_reli} with $C_\oplus = 5$ is fulfilled on all the considered meshes. The distribution of the marked cells presented in Figures~\ref{fig_E4_064} and \ref{fig_E4_256} provides the visual indication that the expected layers are accurately detected by the error estimator.
For this example, the timings presented in Table~\ref{nexADtime} show that, unlike the previous examples, assembling and solving the system for the estimator is faster than for the original problem not only in Case~\ref{ex:case_3} (less than $1/2$ of the original cost), but also in Case~\ref{ex:case_2} (about $2/3$ of the original cost). This is due to the SUPG stabilization which is costlier than computing the additional term $b\cdot \nabla u_h$ in the majorant $M_{\oplus,\text{adv}}^{2}$.

\section{Conclusion}
\label{sec_Conc}

We have proposed a method for cost-efficient computation of guaranteed and sharp a posteriori error estimates in IGA. This method relies only on the use of NURBS basis functions, without the need for constructing complicated basis functions in $H(\Omega,\dvg)$.
We have discussed different settings which allow the user to balance the sharpness of the bound and accurate error distribution on the one hand, and the required computational cost of the error estimator on the other hand (see Remark~\ref{rmk_balance}). For the presented settings, we have derived a quality criterion, which is easy to check numerically and which indicates whether the computed estimate is sharp or not (see Remark~\ref{rmk_Coplus}).
Two properties of NURBS basis functions are exploited. Firstly, the basis functions are, in general, automatically in $H(\Omega,\dvg)$ due to their high smoothness. Without this property, we could not use NURBS of equal degree for both components as basis functions for the minimizing function $y_h$. Secondly, increasing the polynomial degree of NURBS basis functions adds only few DOFs. This fact is necessary for keeping the computational cost of the majorant as low as possible (see Remark~\ref{rem_choiceYh}). It is important to note that none of these properties are possible in FEM discretizations based on $C^{0}$ basis functions.
Apart from the topical interest of a posteriori error estimation and adaptivity, the presented method should also be of interest in parametrization of computational domain. For example, for $r$- refinement in IGA, i.e., to optimize the placement of inner control points, the proposed estimator can be used to accurately detect the regions with large error and then use the optimization algorithm to reposition the control points. Such a problem of $r$- refinement has been studied in \cite{Xu11_2021, XuMDG-13}.
Finally, in this paper, we have only considered tensor-product NURBS discretizations. While the extension of this method to locally refined isogeometric discretizations and also to three dimensions is, in theory, straightforward, the actual performance and efficiency of the error estimator on such methods and meshes is the subject of further studies.

\section*{Acknowledgements}
The authors are grateful to Prof.\ Sergey I.\ Repin, V.A.~Steklov Institute of Mathematics, St.~Petersburg, for helpful discussions. Authors are also thankful to unknown referees for their helpful comments. The support from the Austrian Science Fund (FWF) through the project P21516-N18 is gratefully acknowledged.


\end{document}